\documentclass[english,13pt]{article}
\usepackage[T1]{fontenc}
\usepackage[latin1]{inputenc}
\usepackage{geometry}
\usepackage{float}
\usepackage{mathrsfs}
\usepackage{amsmath}
\usepackage{amssymb}
\usepackage{graphicx}
\usepackage{subfigure}

\makeatletter

\usepackage{algorithm,algpseudocode}

\usepackage{amsthm}

\usepackage{mathrsfs}

\usepackage{amsfonts}

\usepackage{epsfig}

\usepackage{bm}

\usepackage{mathrsfs}

\usepackage{enumerate}

\numberwithin{equation}{section}

\@ifundefined{definecolor}{\@ifundefined{definecolor}
 {\@ifundefined{definecolor}
 {\usepackage{color}}{}
}{}
}{}

\usepackage{subfig}\usepackage[all]{xy}

\newtheorem{theorem}{Theorem}[section]
\newtheorem{lem}{Lemma}[section]
\newtheorem{rem}{Remark}[section]
\newtheorem{prop}{Proposition}[section]

\newcounter{hypA}
\newenvironment{hypA}{\refstepcounter{hypA}\begin{itemize}
  \item[({\bf A\arabic{hypA}})]}{\end{itemize}}
\newcounter{hypB}

\newcounter{hypD}

\usepackage{babel}\date{}

\usepackage{babel}

\makeatother

\usepackage{babel}

\begin{document}

\begin{center}

{\Large \textbf{Antithetic Multilevel Particle Filters}}

\vspace{0.5cm}

BY  AJAY JASRA$^{1}$, MOHAMED MAAMA$^{1}$ \& HERNANDO OMBAO$^{2}$

{\footnotesize $^{1}$Applied Mathematics and Computational Science Program, \footnotesize $^{2}$Statistics Program, Computer, Electrical and Mathematical Sciences and Engineering Division, King Abdullah University of Science and Technology, Thuwal, 23955-6900, KSA.}
{\footnotesize E-Mail:\,} \texttt{\emph{\footnotesize ajay.jasra@kaust.edu.sa, maama.mohamed@gmail.com, hernando.ombao@kaust.edu.sa}}

\end{center}

\begin{abstract}
In this paper we consider the filtering of partially observed multi-dimensional diffusion processes that are observed regularly at discrete times. This is a challenging problem which requires the use of advanced numerical schemes based upon time-discretization of the diffusion process and then the application of particle filters. Perhaps the state-of-the-art method for moderate dimensional problems is the multilevel particle filter of \cite{mlpf}. This is a method that combines multilevel Monte Carlo and particle filters. The approach in that article is based intrinsically upon an Euler discretization method. We develop a new particle filter based upon the antithetic truncated Milstein scheme of \cite{ml_anti}. We show that for a class of diffusion problems, for $\epsilon>0$ given, that the cost to produce a mean square error (MSE) in estimation of the filter, of $\mathcal{O}(\epsilon^2)$ is $\mathcal{O}(\epsilon^{-2}\log(\epsilon)^2)$. In the case of multidimensional diffusions with non-constant diffusion coefficient, the method of \cite{mlpf} has a cost of $\mathcal{O}(\epsilon^{-2.5})$ to achieve the same MSE. We support our theory with numerical results in several examples. 
\end{abstract}

\section{Introduction}

We are given a diffusion process:
\begin{equation}\label{eq:diff_proc}
dX_t = \alpha(X_t)dt + \beta(X_t)dW_t
\end{equation}
where $X_0=x_0\in\mathbb{R}^d$ is given, $\alpha:\mathbb{R}^d\rightarrow\mathbb{R}^d$, $\beta:\mathbb{R}^d\rightarrow\mathbb{R}^{d\times d}$ and $\{W_t\}_{t\geq 0}$ is a
standard $d-$dimensional Brownian motion. We consider the problem where this is a latent process and we observe it only through a sequence of discrete and regular in time data and in particular where the structure of the observations are a special case of a state-space or hidden Markov model (HMM). This class of models has a wide class of applications from finance, econometrics and engineering; see \cite{cappe,delmoral1} for example.

In this paper we consider the problem of filtering and normalizing constant (marginal likelihood) estimation. For standard HMMs, that is, where the latent process is a discrete-time Markov chain, this is a notoriously challenging problem, requiring the application of advanced numerical (Monte Carlo) methods such as the particle filter (PF); see e.g.~\cite{cappe,delmoral1} for a survey. The scenario which we consider is even more challenging as typically the transition density associated to \eqref{eq:diff_proc}, assuming it exists, is often intractable and this can limit the application of PFs; although there are some exceptions \cite{fearn}, and exact simulation methods \cite{beskos1,blanchet} these are often not general enough or too expensive to be of practical use in the filtering problem. As a result, we focus on the case where one discretizes the process \eqref{eq:diff_proc} in time.

In recent years, one of the most successful methods for improving Monte Carlo based estimators, associated to probability laws under time discretization is the multilevel Monte Carlo (MLMC) method \cite{giles,giles1,hein}. This is an approach that considers a collapsing sum representation of an expectation w.r.t.~a probability law at given level of discretization. The collapsing element is associated to differences in expectations with increasingly coarse discretization levels with a final (single) expectation at a course level. Then if one can sample appropriate couplings of the probability laws at consecutive levels, then it is possible to reduce the cost to achieve a certain mean square error (MSE) in several examples, diffusions being one of them. This method has been combined with the PF in several articles, resulting in the multilevel particle filter (MLPF); see \cite{mlpf,mlpf1,levymlpf} and also \cite{ml_rev} for a review and \cite{ub_pf} for extensions. 

The method of \cite{mlpf} is intrinsically based upon the ability to sample couplings of discretized diffusion processes.
In almost all of the applications that we are aware of, this is based upon the synchronous coupling of Brownian motion for an Euler or Milstein scheme. These particular couplings inherit properties of the strong error of the associated time discretization. The importance of the strong error rate is that it can help determine the efficiency gain of any MLMC approach of which the MLPF is one. As is well-known the Milstein scheme, which is of higher (strong) order than the Euler method, in dimensions ($d> 2$) of two or larger can be difficult to implement numerically due to the need of simulating L\'evy areas. \cite{ml_anti} consider computing expectations associated to the law of  \eqref{eq:diff_proc}
at a given terminal time. They show 
that when eliminating the L\'evy area and including an antithetic type modification of the traditional 
MLMC estimator, one can maintain the strong error rate of the Milstein scheme even in multi-dimensions. Moreover the simulation is of a cost that is of a same order as Euler discretizations which are the ones that would (most often) be used in multi-dimensional cases.

We develop a new particle filter based upon the antithetic truncated Milstein scheme of \cite{ml_anti}. We show that for a class of diffusion problems, for $\epsilon>0$ given, that the cost to produce a MSE in estimation of the filter, of $\mathcal{O}(\epsilon^2)$ is $\mathcal{O}(\epsilon^{-2}\log(\epsilon)^2)$. In the case of multi-dimensional diffusions with non-constant diffusion coefficient, the method of \cite{mlpf} has a cost of $\mathcal{O}(\epsilon^{-2.5})$ to achieve the same MSE. We also show how this new PF can be used to compute the normalizing constant recursively in time, but we do not prove anything on the efficiency gains. All of our theory is confirmed in several numerical examples.

This article is structured as follows. In Section \ref{sec:model} we give details on the model to be considered and the time-discretization associated to the process \eqref{eq:diff_proc}. In Section \ref{sec:algo} we present our algorithm.
Section \ref{sec:math} details our mathematical results and their algorithmic implications. In Section \ref{sec:numerics}
we give numerical results that support our theory. Most of our mathematical proofs can be found in the appendices at the end of the article.

\section{Model and Discretization}\label{sec:model}

\subsection{State-Space Model}

Our objective is to consider the filtering problem and normalizing constant estimation for a specific class of state-space models associated to the diffusion process \eqref{eq:diff_proc}. In particular, we will assume that \eqref{eq:diff_proc} is subject to a certain assumption (A\ref{ass:diff1}) which described in the Appendix, Section \ref{app:mil}. This assumption is certainly strong enough to guarantee that \eqref{eq:diff_proc} has a unique solution and in addition that the diffusion has a transition probability which we denote, over 1 unit time, as $P(x,dx')$. Then we are interested in filtering associated to the state-space model
$$
p(dx_{1:n},y_{1:n}) = \prod_{k=1}^n P(x_{k-1},dx_k)g(x_k,y_k)
$$
where $y_{1:n}=(y_1,\dots,y_n)^{\top}\in\mathsf{Y}^n$ are observations each with conditional density $g(x_k,\cdot)$. The filter associated to this measure is for $k\in\mathbb{N}$
$$
\pi_k(dx_k) = \frac{\int_{\mathsf{X}^{k-1}} p(dx_{1:k},y_{1:k})}{\int_{\mathsf{X}^{k}} p(dx_{1:k},y_{1:k})}
$$
where $\mathsf{X}=\mathbb{R}^d$. The denominator 
$$
p(y_{1:k}) := \int_{\mathsf{X}^{k}} p(dx_{1:k},y_{1:k})
$$
is the normalizing constant or marginal likelihood. This latter object is often used in statistics for model selection.

In practice, we assume that working directly with $P$ is not possible, either due to intractability or cost of simulation. We will propose to work with an alternative collection of filters based upon time discretization which is what we now describe.

\subsection{Time Discretization}

Typically one must time discretize \eqref{eq:diff_proc} and we consider a time discretization at equally spaced times, separated by $\Delta_l=2^{-l}$. To continue with our exposition, we define the $d-$vector, $H:\mathbb{R}^{2d}\times\mathbb{R}^+\rightarrow\mathbb{R}^d$, 
$H_{\Delta}(x,z)=(H_{\Delta,1}(x,z),\dots,H_{\Delta,d}(x,z))^{\top}$ where for $i\in\{1,\dots,d\}$
\begin{eqnarray*}
H_{\Delta,i}(x,z) & = &  \sum_{(j,k)\in\{1,\dots,d\}^2} h_{ijk}(x)(z_jz_k-\Delta) \\
h_{ijk}(x) & = & \frac{1}{2}\sum_{m\in\{1,\dots,d\}} \beta_{mk}(x)\frac{\partial \beta_{ij}(x)}{\partial x_m}.
\end{eqnarray*}
We denote by $\mathcal{N}_d(\mu,\Sigma)$ the $d-$dimensional Gaussian distribution with mean vector $\mu$
and covariance matrix $\Sigma$; if $d=1$ we drop the subscript $d$. $I_d$ is the $d\times d$ identity matrix.
A single level version of the truncated Milstein scheme, which is the focus of this article, is presented in Algorithm \ref{alg:milstein_sl}.
Ultimately, we will consider the antithetic truncated Milstein scheme in \cite{ml_anti},
which can be simulated as described in Algorithm \ref{alg:milstein}. The method in Algorithm \ref{alg:milstein} is a means to approximate differences of expectations w.r.t.~discretized laws at consecutive levels as is used in MLMC and this latter approach will be detailed in Section \ref{sec:algo}.

We denote the transition kernel induced by Algorithm \ref{alg:milstein_sl} as $P^l(x,dx')$ and for a given $l\in\mathbb{N}$ with  we will be concerned with the filter induced by the following joint measure
$$
p^l(dx_{1:n},y_{1:n}) = \prod_{k=1}^n P^l(x_{k-1},dx_k) g(x_k,y_k).
$$
The filter associated to this measure is for $k\in\mathbb{N}$
$$
\pi_k^l(dx_k) = \frac{\int_{\mathsf{X}^{k-1}} p^l(dx_{1:k},y_{1:k})}{\int_{\mathsf{X}^{k}} p^l(dx_{1:k},y_{1:k})}.
$$
The associated normalizing constant is
$$
p^l(y_{1:k}) = \int_{\mathsf{X}^{k}} p^l(dx_{1:k},y_{1:k}).
$$
Ultimately, we will seek to approximate expectations w.r.t.~$\pi_k^l$ and to estimate $p^l(y_{1:k})$ as $k$ increases and the study the MSE of a numerical approximation, relative to considering the exact filter  $\pi_k$ and associated normalizing constant $p(y_{1:k})$.

\begin{algorithm}[H]
\begin{enumerate}
\item{Input level $l$ and starting point $x_0^l$.}
\item{Generate $Z_k\stackrel{\textrm{i.i.d.}}{\sim}\mathcal{N}_d(0,\Delta_l I_d)$, $k\in\{1,2,\dots,\Delta_l^{-1}\}$.}
\item{Generate level $l$: for $k\in\{0,1,\dots,\Delta_l^{-1}-1\}$ with $X_0^l=x_0^l$
$$
X_{(k+1)\Delta_l}^l = X_{k\Delta_l}^l + \alpha(X_{k\Delta_l}^l)\Delta_l + \beta(X_{k\Delta_l}^l)Z_{k+1} + H_{\Delta_l}(X_{k\Delta_l}^l,Z_{k+1}).
$$
}
\item{Output $X_1^l$.}
\end{enumerate}
\caption{Truncated Milstein Scheme on $[0,1]$.}
\label{alg:milstein_sl}
\end{algorithm}

\begin{algorithm}[H]
\begin{enumerate}
\item{Input level $l$ and starting points $(x_0^l,x_0^{l-1},x_0^{l,a})$.}
\item{Generate $Z_k\stackrel{\textrm{i.i.d.}}{\sim}\mathcal{N}_d(0,\Delta_l I_d)$, $k\in\{1,2,\dots,\Delta_l^{-1}\}$.}
\item{Generate level $l$: for $k\in\{0,1,\dots,\Delta_l^{-1}-1\}$ with $X_0^l=x_0^l$
$$
X_{(k+1)\Delta_l}^l = X_{k\Delta_l}^l + \alpha(X_{k\Delta_l}^l)\Delta_l + \beta(X_{k\Delta_l}^l)Z_{k+1} + H_{\Delta_l}(X_{k\Delta_l}^l,Z_{k+1}).
$$
}
\item{Generate level $l-1$: for $k\in\{0,1,\dots,\Delta_{l-1}^{-1}-1\}$ with $X_0^{l-1}=x_0^{l-1}$
\begin{eqnarray*}
X_{(k+1)\Delta_{l-1}}^{l-1} & = & X_{k\Delta_{l-1}}^{l-1} + \alpha(X_{k\Delta_{l-1}}^{l-1})\Delta_{l-1} + \beta(X_{k\Delta_{l-1}}^{l-1})\{Z_{2(k+1)-1}+Z_{2(k+1)}\} +\\ & & H_{\Delta_{l-1}}(X_{k\Delta_{l-1}}^{l-1},Z_{2(k+1)-1}+Z_{2(k+1)}).
\end{eqnarray*}
}
\item{Generate antithetic level $l$: for $k\in\{0,1,\dots,\Delta_l^{-1}-1\}$ with $X_0^{l,a}=x_0^{l,a}$
$$
X_{(k+1)\Delta_l}^{l,a} = X_{k\Delta_l}^{l,a} + \alpha(X_{k\Delta_l}^{l,a})\Delta_l + \beta(X_{k\Delta_l}^{l,a})Z_{\rho_k} + H_{\Delta_l}(X_{k\Delta_l}^{l,a},Z_{\rho_{k}})
$$
where $\rho_k=k+2\mathbb{I}_{\{0,2,4,\dots\}}(k)$.}
\item{Output $(X_1^l,X_1^{l-1},X_1^{l,a})$.}
\end{enumerate}
\caption{Antithetic Truncated Milstein Scheme on $[0,1]$.}
\label{alg:milstein}
\end{algorithm}

\section{Antithetic Multilevel Particle Filters}\label{sec:algo}

\subsection{Multilevel Particle Filter}

Let $\varphi\in\mathcal{B}_b(\mathsf{X})$ with the latter the collection of bounded and measurable real-valued functions: we write $\pi_k^l(\varphi)=\int_{\mathsf{X}}\varphi(x_k)\pi_k^l(dx_k)$. Let $(\underline{L},\overline{L})\in\mathbb{N}_0\times\mathbb{N}$, with $\underline{L}<\overline{L}$, be given, our objective is the approximation of $\pi_k^{\overline{L}}(\varphi)$ sequentially in time. This can be achieved using the MLMC identity
\begin{equation}\label{eq:mlmc_id}
\pi_k^{\overline{L}}(\varphi) = \pi_k^{\underline{L}}(\varphi) + \sum_{l=\underline{L}+1}^{\overline{L}} [\pi_{k}^l-\pi_{k}^{l-1}](\varphi).
\end{equation}
As noted in the introduction, the computational cost of approximating the R.H.S.~of \eqref{eq:mlmc_id} versus the L.H.S.~can be lower, when seeking to achieve a pre-specified MSE. The R.H.S.~of \eqref{eq:mlmc_id} can be approximated by the method in \cite{mlpf}, but we shall consider a modification which can improve on the computational effort to achieve a given mean square error.

We begin first with approximating $\pi_k^{\underline{L}}(\varphi)$. This can be approximated using the standard PF as described in Algorithm \ref{alg:pf}. It is by now standard in the literature that 
$$
\pi_k^{N_{\underline{L}},\underline{L}}(\varphi) := \sum_{i=1}^{N_{\underline{L}}}\frac{g(X_k^{i,\underline{L}},y_k)}{\sum_{j=1}^{N_{l}} g(X_k^{j,\underline{L}},y_k)}\varphi(X_k^{i,\underline{L}})
$$
will converge almost surely to $\pi_k^{\underline{L}}(\varphi)$ where the samples $X_k^{i,\underline{L}}$ are after Step 1.~or 3.~of Algorithm \ref{alg:pf}.

\begin{algorithm}[h]
\begin{enumerate}
\item{Initialization: For $i\in\{1,\dots,N_{\underline{L}}\}$, generate $X_1^{i,\underline{L}}$ independently using Algorithm \ref{alg:milstein_sl} with level $\underline{L}$ and starting point $x_0$. Set $k=1$.}
\item{Resampling: Compute 
\begin{equation}\label{eq:res_weights}
\left(\frac{g(x_k^{1,\underline{L}},y_k)}{\sum_{j=1}^{N_{\underline{L}}} g(x_k^{j,\underline{L}},y_k)},\dots,\frac{g(x_k^{N_{\underline{L}},\underline{L}},y_k)}{\sum_{j=1}^{N_{\underline{L}}} g(x_k^{j,\underline{L}},y_k)}\right).
\end{equation}
For $i\in\{1,\dots,N_{\underline{L}}\}$ generate an index $a_k^i$ using the probability mass function in \eqref{eq:res_weights} and set $\tilde{X}_k^{i,\underline{L}}=X_k^{a_k^i,\underline{L}}$.
Then set $X_k^{1:N_{\underline{L}},\underline{L}}=\tilde{X}_k^{1:N_{\underline{L}},\underline{L}}$.
}
\item{Sampling:  For $i\in\{1,\dots,N_{\underline{L}}\}$, generate $X_{k+1}^{i,\underline{L}}|X_k^{1:N_{\underline{L}},\underline{L}}$ conditionally independently using Algorithm \ref{alg:milstein_sl} with level $\underline{L}$ and starting point $x_k^{i,\underline{L}}$. Set $k=k+1$.}
\end{enumerate}
\caption{Particle Filter at Level $\underline{L}$.}
\label{alg:pf}
\end{algorithm}

Approximating the differences $[\pi_{k}^l-\pi_{k}^{l-1}](\varphi)$ takes a little more work and to that end, we present a new resampling method in \ref{alg:max_coup} which we shall employ. This resampling method gives rise to the new coupled particle filter which is described in Algorithm \ref{alg:coup_pf}. In this context one can define for $(k,l,N_l,\varphi)\in\mathbb{N}^3\times\mathcal{B}_b(\mathsf{X})$:
\begin{eqnarray}
\pi_k^{N_l,l}(\varphi) & := & \sum_{i=1}^{N_{l}} W_k^{i,l}\varphi(X_k^{i,l}) \label{eq:pi_l_def}\\
\bar{\pi}_k^{N_l,l-1}(\varphi) & := & \sum_{i=1}^{N_{l}} \bar{W}_k^{i,l-1}\varphi(X_k^{i,l-1}) \nonumber\\
\pi_k^{N_l,l,a}(\varphi) & := & \sum_{i=1}^{N_{l}} W_k^{i,l,a}\varphi(X_k^{i,l,a}) \nonumber
\end{eqnarray}
and estimate $[\pi_{k}^l-\pi_{k}^{l-1}](\varphi)$  as
$$
[\pi_{k}^l-\pi_{k}^{l-1}]^{N_l}(\varphi) := \frac{1}{2}\{\pi_k^{N_l,l}(\varphi) + \pi_k^{N_l,l,a}(\varphi)\} - \bar{\pi}_k^{N_l,l-1}(\varphi)
$$
The samples in $(\pi_k^{N_l,l}(\varphi), \pi_k^{N_l,l,a}(\varphi),\bar{\pi}_k^{N_l,l-1}(\varphi))$ are obtained after Step 1.~or 3.~in Algorithm \ref{alg:coup_pf}.

To conclude the antithetic multilevel particle filter (AMLPF) is then as follows:
\begin{enumerate}
\item{At level $\underline{L}$ run  Algorithm \ref{alg:pf} sequentially in time.}
\item{At level $l\in\{\underline{L}+1, \underline{L}+2,\dots,\overline{L}\}$ independently of all other $l$ and of Step 1.~run  Algorithm \ref{alg:pf} sequentially in time.}
\end{enumerate}
To approximate the R.H.S.~of \eqref{eq:mlmc_id} at any time $k$, one can use the estimator:
$$
\pi_k^{\overline{L},ML}(\varphi) := \pi_k^{N_{\underline{L}},\underline{L}}(\varphi) + \sum_{l=\underline{L}+1}^{\overline{L}}[\pi_{k}^l-\pi_{k}^{l-1}]^{N_l}(\varphi)
$$
where $\pi_k^{N_{\underline{L}},\underline{L}}(\varphi)$ has the definition \eqref{eq:pi_l_def} when $l=\underline{L}$ and has been obtained using  Algorithm \ref{alg:pf}.

The normalizing constant $p^{\overline{L}}(y_{1:k})$ can also be approximated using the AMLPF. We make the definitions for any 
$(k,l,N_l,\varphi)\in\mathbb{N}^3\times\mathcal{B}_b(\mathsf{X})$:
\begin{eqnarray}
\eta_k^{N_l,l}(\varphi) & := & \sum_{i=1}^{N_{l}} \varphi(X_k^{i,l}) \label{eq:eta_l_def}\\
\bar{\eta}_k^{N_l,l-1}(\varphi) & := & \sum_{i=1}^{N_{l}} \varphi(X_k^{i,l-1}) \nonumber\\
\eta_k^{N_l,l,a}(\varphi) & := & \sum_{i=1}^{N_{l}} \varphi(X_k^{i,l,a}).\nonumber
\end{eqnarray}
Then we have the following approximation of 
$$
p^{\overline{L},ML}(y_{1:k}) := \prod_{j=1}^{k}\eta_j^{N_{\underline{L}},\underline{L}}(g_j) + \sum_{l=\underline{L}+1}^{\overline{L}}\Big\{\tfrac{1}{2}\prod_{j=1}^{k}\eta_j^{N_l,l,}(g_j) + 
\tfrac{1}{2}\prod_{j=1}^{k}\eta_j^{N_l,l,a}(g_j) -
\prod_{j=1}^{k}\bar{\eta}_j^{N_{l-1},l-1,}(g_j)
\Big\}
$$
where $g_j(x)=g(x,y_j)$ is used as a short-hand notation and $g_j$ is assumed bounded and measurable for each $j\in\mathbb{N}$ and $\eta_k^{N_{\underline{L}},\underline{L}}(\varphi)$ has the definition \eqref{eq:eta_l_def} when $l=\underline{L}$ and has been obtained using  Algorithm \ref{alg:pf}.. This estimator is unbiased in terms of $p^{\overline{L}}(y_{1:k})$,  (see e.g.~\cite{delmoral1}) 
and is consistent as the number of samples used at each level grows.


\begin{algorithm}[h]
\begin{enumerate}
\item{Input: $(U_1^{1:N},U_2^{1:N},U_3^{1:N})$ and probabilities $(W_1^{1:N},W_2^{1:N},W_3^{1:N})$.}
\item{For $i\in\{1,\dots,N\}$ generate $U\sim\mathcal{U}_{[0,1]}$ (uniform distribution on $[0,1]$)
\begin{itemize}
\item{If $U<\sum_{i=1}^N \min\{W_1^i,W_2^i,W_3^i\}$ generate $a^i\in\{1,\dots,N\}$ using the probability mass function
$$
\mathbb{P}(i) = \frac{\min\{W_1^i,W_2^i,W_3^i\}}{\sum_{j=1}^N \min\{W_1^j,W_2^j,W_3^j\}}
$$
and set $\tilde{U}_j^i=U_j^{a^i}$, 
$j\in\{1,2,3\}$.}
\item{Otherwise generate $(a_1^i,a_2^i,a_3^i)\in\{1,\dots,N\}^3$ using any coupling of the probability mass functions:
$$
\mathbb{P}_j(i) = \frac{W_j^i-\min\{W_1^i,W_2^i,W_3^i\}}{\sum_{k=1}^N[W_j^k-\min\{W_1^k,W_2^k,W_3^k\}]}
$$
and set $\tilde{U}_j^i=U_j^{a_j^i}$,  $j\in\{1,2,3\}$.}
\end{itemize}
}
\item{Set: $U_j^i=\tilde{U}_j^{i}$, $(i,j)\in\{1,\dots,N\}\times\{1,2,3\}$.}
\item{Output: $(U_1^{1:N},U_2^{1:N},U_3^{1:N})$.}
\end{enumerate}
\caption{Maximal Coupling-Type Resampling}
\label{alg:max_coup}
\end{algorithm}

\begin{algorithm}[h]
\begin{enumerate}
\item{Initialization: For $i\in\{1,\dots,N_{\underline{l}}\}$, generate $U_1^{i,l} = (X_1^{i,l},\bar{X}_1^{i,l-1},X_1^{i,l,a})$ independently using Algorithm \ref{alg:milstein} with level $l$ and starting points $(x_0^{l},\bar{x}_0^{l-1},x_0^{l,a})$. Set $k=1$.}
\item{Coupled Resampling: Compute 
\begin{eqnarray*}
W_k^{1:N_l,l} & = &  \left(\frac{g(x_k^{1,l},y_k)}{\sum_{j=1}^{N_{l}} g(x_k^{j,l},y_k)},\dots,\frac{g(x_k^{N_l,l},y_k)}{\sum_{j=1}^{N_{l}} g(x_k^{j,l},y_k)}\right) \\
\bar{W}_k^{1:N_l,l-1} & = &  \left(\frac{g(\bar{x}_k^{1,l-1},y_k)}{\sum_{j=1}^{N_{l}} g(\bar{x}_k^{j,l-1},y_k)},\dots,\frac{g(\bar{x}_k^{N_l,l-1},y_k)}{\sum_{j=1}^{N_{l}} g(\bar{x}_k^{j,l-1},y_k)}\right) \\
W_k^{1:N_l,l,a} & = & \left(\frac{g(x_k^{1,l,a},y_k)}{\sum_{j=1}^{N_{l}} g(x_k^{j,l,a},y_k)},\dots,\frac{g(x_k^{N_l,l,a},y_k)}{\sum_{j=1}^{N_{l}} g(x_k^{j,l,a},y_k)}\right).
\end{eqnarray*}
Then using
$(X_k^{1:N_l,l},\bar{X}_k^{1:N_l,l-1}, X_k^{1:N_l,l,a})$ and $(W_k^{1:N_l,l},\bar{W}_k^{1:N_l,l-1},W_k^{1:N_l,l,a})$ call Algorithm \ref{alg:max_coup}.
}
\item{Coupled Sampling: For $i\in\{1,\dots,N_{\underline{L}}\}$, generate $U_{k+1}^{i,l}|U_k^{1:N_{l},l}$ conditionally independently using Algorithm \ref{alg:milstein} with level $l$ and starting point $(x_k^{i,l},\bar{x}_k^{i,l-1},x_k^{i,l,a})$. Set $k=k+1$.}
\end{enumerate}
\caption{A New Coupled Particle Filter for $l\in\mathbb{N}$ given.}
\label{alg:coup_pf}
\end{algorithm}

\section{Mathematical Results}\label{sec:math}

We now present our main mathematical result which relates to the $\mathbb{L}_2-$convergence of our estimator of the filter. We remark that we will consider the estimation of the normalizing constant in a sequel paper.
The result is proved under two assumptions (A\ref{ass:diff1}) and (A\ref{ass:g}), which are stated in the Appendix; see Appendix \ref{app:intro} for more information. $\mathcal{C}_b^2(\mathsf{X},\mathbb{R})$ denotes the collection of twice continuously differentiable functions from $\mathsf{X}$ to $\mathbb{R}$ with bounded derivatives of all order 1 and 2. For $\varphi\in\mathcal{B}_b(\mathsf{X})$
we use the notation $\pi_k^{\overline{L}}(\varphi)=\int_{\mathsf{X}}\varphi(x)\pi_k^{\overline{L}}(dx)$.
$\mathbb{E}$ denotes the expectation w.r.t.~the law used to generate the AMLPF.

\begin{theorem}\label{theo:main_theo}
Assume (A\ref{ass:diff1}-\ref{ass:g}). Then for any $(k,\varphi)\in\mathbb{N}\times
\mathcal{B}_b(\mathsf{X})\cap\mathcal{C}_b^2(\mathsf{X},\mathbb{R})$ there exists a $C<+\infty$ such that for any $(\underline{L},\overline{L},l,N_{\underline{L}},\dots,N_{\overline{L}},\varepsilon)\in\mathbb{N}_0^2\times\{\underline{L},\dots,\overline{L}\}\times
\mathbb{N}^{\overline{L}-\underline{L}+1}
\times(0,1/2)$ with  $\underline{L}<\overline{L}$
$$
\mathbb{E}\left[\left(\pi_k^{\overline{L},ML}(\varphi)-\pi_k^{\overline{L}}(\varphi)\right)^2\right]
 \leq 
C\Bigg(\frac{1}{N_{\underline{L}}} + 
\sum_{l=\underline{L}+1}^{\overline{L}}\Bigg\{
\frac{\Delta_l}{N_l} +
\frac{\Delta_l^{\tfrac{1}{2}}}{N_l^2}
+ \frac{\Delta_l^{\tfrac{1}{2}-\varepsilon}}{N_l^{2}}
\Bigg\}
+ 
 $$
 $$
\sum_{(l,q)\in\mathsf{D}_{\underline{L},\overline{L}}}\Bigg\{
\frac{1}{N_l}\Bigg\{
\Bigg(\Delta_l
+  \frac{\Delta_l^{\tfrac{1}{2}-\varepsilon}}{N_l}
\Bigg)^{\tfrac{1}{2}} + 
\left(\Delta_l^{\tfrac{1}{2}}
+ \frac{\Delta_l^{\tfrac{1}{2}-\varepsilon}}{N_l}
\right)
\Bigg\}
\frac{1}{N_q}\Bigg\{\left(\Delta_l
+ \frac{\Delta_q^{\tfrac{1}{2}-\varepsilon}}{N_q}
\right)^{\tfrac{1}{2}} +
\left(\Delta_q^{\tfrac{1}{2}}
+ \frac{\Delta_q^{\tfrac{1}{2}-\varepsilon}}{N_q}
\right)
\Bigg\}
\Bigg\}
\Bigg)
$$
where $\mathsf{D}_{\underline{L},\overline{L}}=\{(l,q)\in\{\underline{L},\dots,\overline{L}\}:l\neq q\}$.
\end{theorem}

\begin{proof}
This follows by using the $C_2-$inequality to separate the terms $\pi_k^{N_{\underline{L}},\underline{L}}(\varphi) -\pi_k^{\underline{L}}(\varphi)$ and $\sum_{l=\underline{L}+1}^{\overline{L}}[\pi_{k}^l-\pi_{k}^{l-1}]^{N_l}(\varphi)-\sum_{l=\underline{L}+1}^{\overline{L}}[\pi_{k}^l-\pi_{k}^{l-1}](\varphi)$. Then one uses standard results for particle
filters (see e.g.~\cite[Lemma A.3]{beskos}) for the $\pi_k^{N_{\underline{L}},\underline{L}}(\varphi) -\pi_k^{\underline{L}}(\varphi)$ and then one multiplies out the brackets and uses Theorems \ref{theo:var} and \ref{theo:bias} for the other term.
\end{proof}

The implication of this result is as follows. We know that the bias $|\pi_k^{\underline{L}}(\varphi)-\pi_k(\varphi)|$ is $\mathcal{O}(\Delta_{\overline{L}})$. This can be proved in a similar manner to \cite[Lemma D.2]{mlpf} along with using the fact that the truncated Milstein scheme is a first order (weak) method. Therefore, for $\epsilon>0$ given, one can choose $\overline{L}$ such that $\Delta_{\overline{L}}=\mathcal{O}(\epsilon)$ (so that the square bias is $\mathcal{O}(\epsilon^2)$. Then to choose $N_{\underline{L}},\dots,N_{\overline{L}}$, one can set
$N_{\underline{L}}=\mathcal{O}(\epsilon^{-2})$ and just as in e.g.~\cite{giles} that
$N_l=\mathcal{O}(\epsilon^{-2}\Delta_l(\overline{L}-\underline{L}+1))$, $l\in\{\underline{L}+1,\dots,\overline{L}\}$. If one does this, one can verify that the upper-bound in Theorem \ref{theo:main_theo} is $\mathcal{O}(\epsilon^2)$; so the MSE is also $\mathcal{O}(\epsilon^2)$. The cost to run the algorithm per-time step is
$\mathcal{O}(\sum_{l=\underline{L}}^{\overline{L}}N_l\Delta_{l}^{-1}) = \mathcal{O}(\epsilon^{-2}\log(\epsilon)^2)$.
In the case that \eqref{eq:diff_proc} has a non-constant diffusion coefficient with $d> 2$, the approach in \cite{mlpf}, based upon using Euler discretizations, 
would obtain a MSE of $\mathcal{O}(\epsilon^2)$ at a cost of $\mathcal{O}(\epsilon^{-2.5})$; a significant increase.

\section{Numerical Results}\label{sec:numerics}

%
\subsection{Models}

We consider three different models for our numerical experiments.

\subsubsection{Model 1: Geometric Brownian motion (GBM) process} 
 Our first model we use is :
$$
dX_t  = \mu X_t dt + \sigma X_t dW_t, \quad X_0=x_0.
$$
We set $Y_n|X_{n}=x \sim \mathcal{N}(\log(x),\tau^2)$ where $\tau^2=0.02$
and $\mathcal{N}(x,\tau^2)$ is the Gaussian distribution with mean $x$ and variance $\tau^2$. We choose $x_0 = 1$, $\mu = 0.02$, and $\sigma = 0.2$.

\subsubsection{Model 2: Clark-Cameron SDE}
 Our second model we consider in this paper is the Clark-Cameron SDE model (e.g.~\cite{ml_anti}) with initial conditions $x_0=(0,0)^{\top}$
\begin{eqnarray*}
dX_{1,t}  & = &  dW_{1,t}    \\ 
dX_{2,t}  & = & X_{1,t}  dW_{2,t}
\end{eqnarray*}
where $X_{j,t}$ denotes the $j^{th}$ dimension of $X_t$, $j\in\{1,\dots,d\}$.
In addition,  $Y_n|X_n=x\sim \mathcal{N}( \tfrac{X_{1} + X_{2}}{2},\tau^2)$ where $\tau^2=0.1$

\subsubsection{Model 3: Multi-dimensional SDEs with a nonlinear diffusion term (NLMs)} 

For our last model we use  the following multi-dimensional SDEs, with $x_0=(0,0)^{\top}$
 \begin{eqnarray*}
dX_{1,t}  & = & \theta_{1} (\mu_{1} - X_{1,t}) dt + \frac{\sigma_1}{\sqrt{1 + X_{1,t}^{2} }}  dW_{1,t}    \\ 
dX_{2,t}  & =  & \theta_{2} (\mu_{2} - X_{1,t}) dt +   \frac{\sigma_2}{\sqrt{1 + X_{1,t}^{2}}}   dW_{2,t}.
\end{eqnarray*}
We set 
$Y_n|X_n=x\sim \mathcal{L}\left(\tfrac{X_{1} + X_{2}}{2},s\right)$ where $\mathcal{L}(m,s)$ is the Laplace distribution with location $m$ and scale $s$. The values of the parameters that we choose are $(\theta_{1}, \theta_{2}) = (1,1)$, $(\mu_{1}, \mu_{2}) = (0,0)$, $(\sigma_{1}, \sigma_{2}) = (1,1)$ and $s= \sqrt{0.1}$.

\subsection{Simulation Settings}

We will consider a comparison of the PF, MLPF (as in \cite{mlpf}) and AMLPF.
For our numerical experiments, we consider multilevel estimators at levels $l = \{3,4,5,6,7\}$. Resampling is performed adaptively. For the PFs, resampling is done when the effective sample size (ESS) is less than $1/2$ of the particle numbers. For the coupled filters, we use the ESS of the coarse filter as the measurement of discrepancy and resampling occurs at the same threshold as the PF. Each simulation is repeated $100$ times.

\subsection{Simulation Results}

We now present our numerical simulations to show the benefits of the AMLPF on the three models under consideration.
The figures we present are the rate of the cost and the mean-squared error for the estimator for the filter and normalizing constant which are Figures \ref{fig:MSEvsCost} and \ref{fig:MSEvsCost-NC}. The figures show that as we increase the levels from $l\in\{3,\dots,7\}$, the difference in the cost between the methods increases. These figures show the advantage and accuracy of using AMLPF.  Table \ref{tab:tab1} presents the estimated rates of MSE with respect to cost for the normalizing constant and filter. The results agree with the theory \cite{mlpf} and the results of this article, which predicts a rate of $-1.5$ for the particle filter,  $-1.25$ for the multilevel particle filter, and $-1$ for the antithetic multilevel particle filter.

\begin{figure}[h!]
\centering
\subfigure{\includegraphics[width=10cm,height=4.5cm]{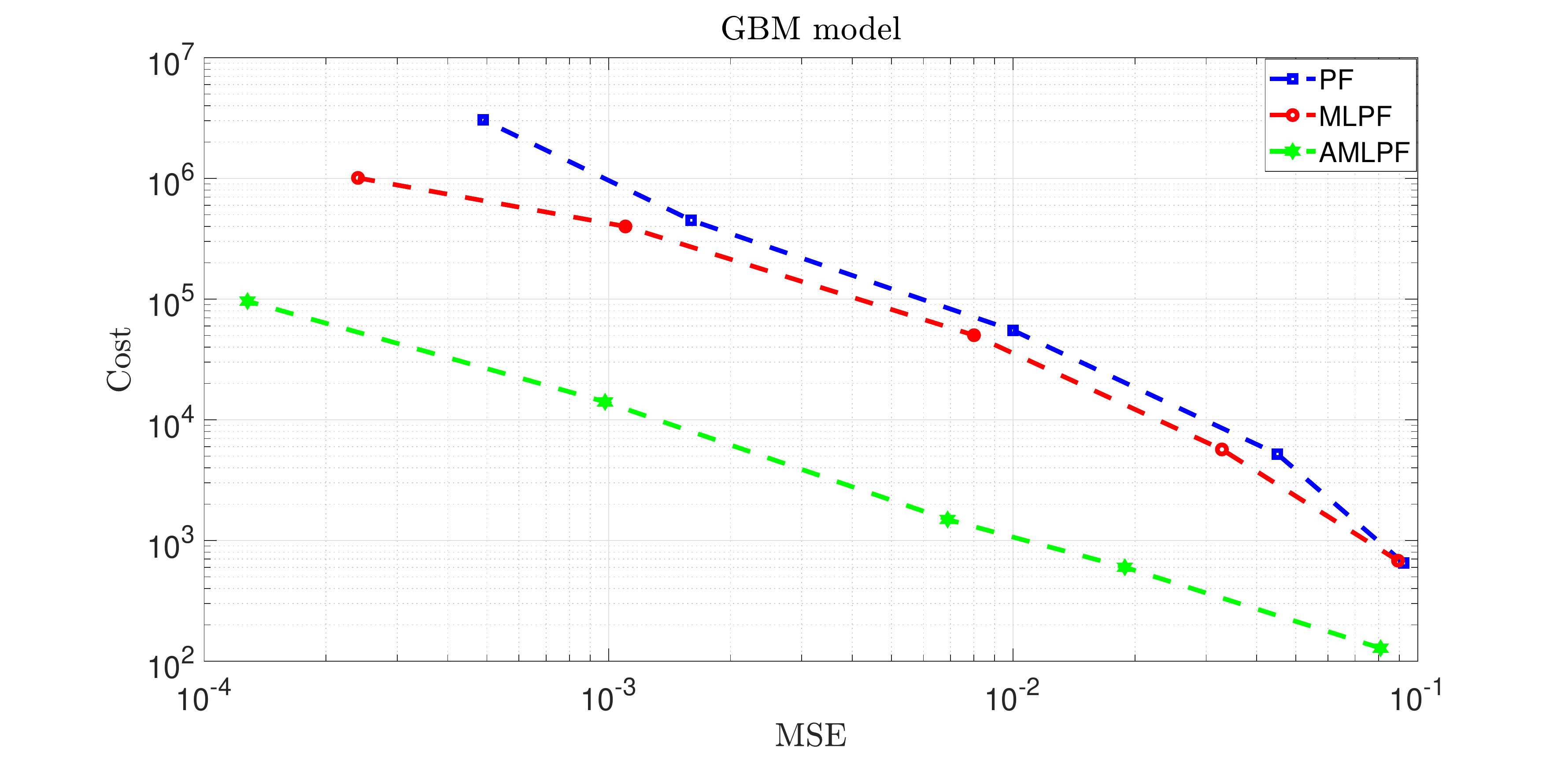}}
\subfigure{\includegraphics[width=10cm,height=4.5cm]{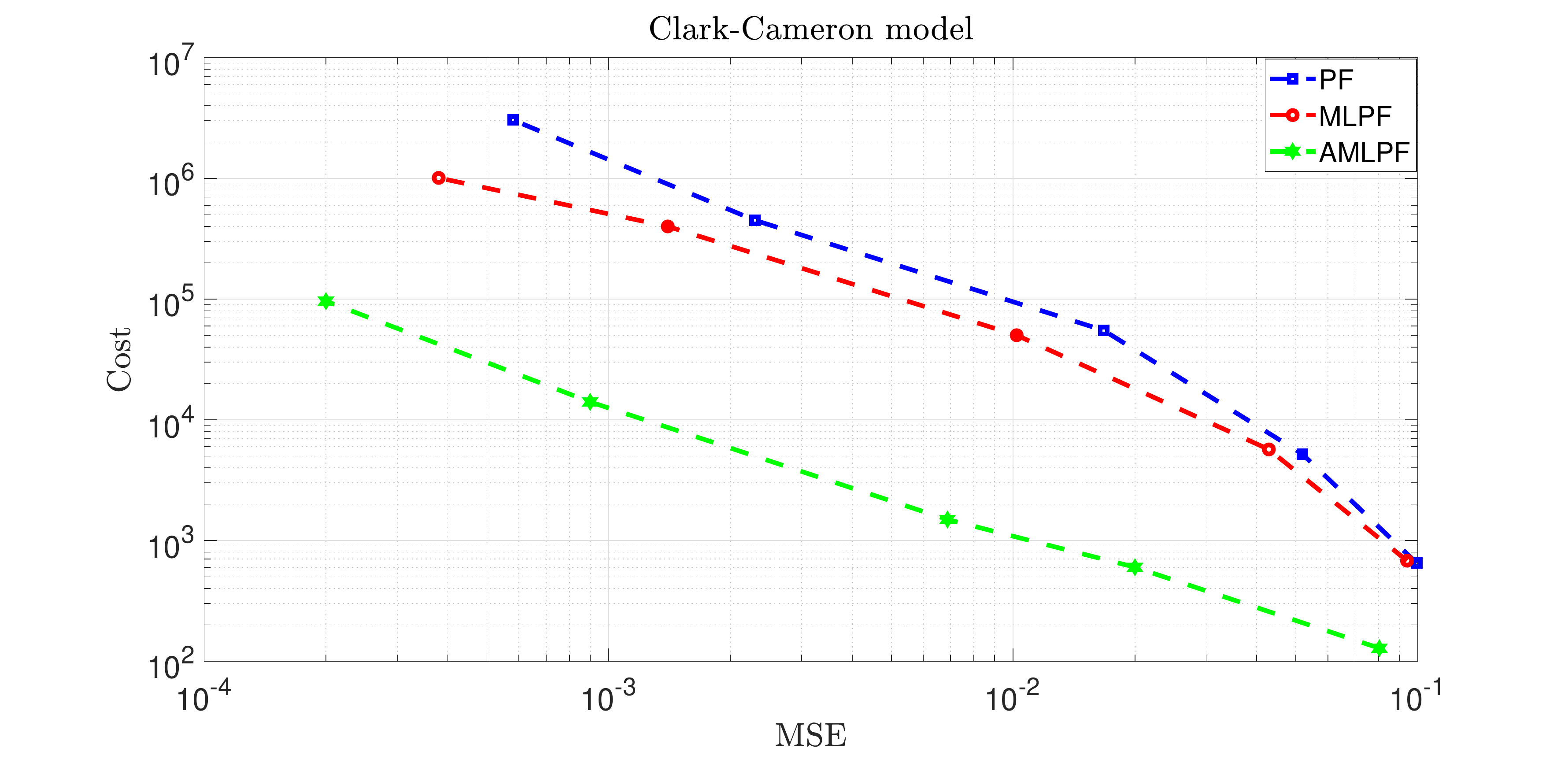}}
\subfigure{\includegraphics[width=10cm,height=4.5cm]{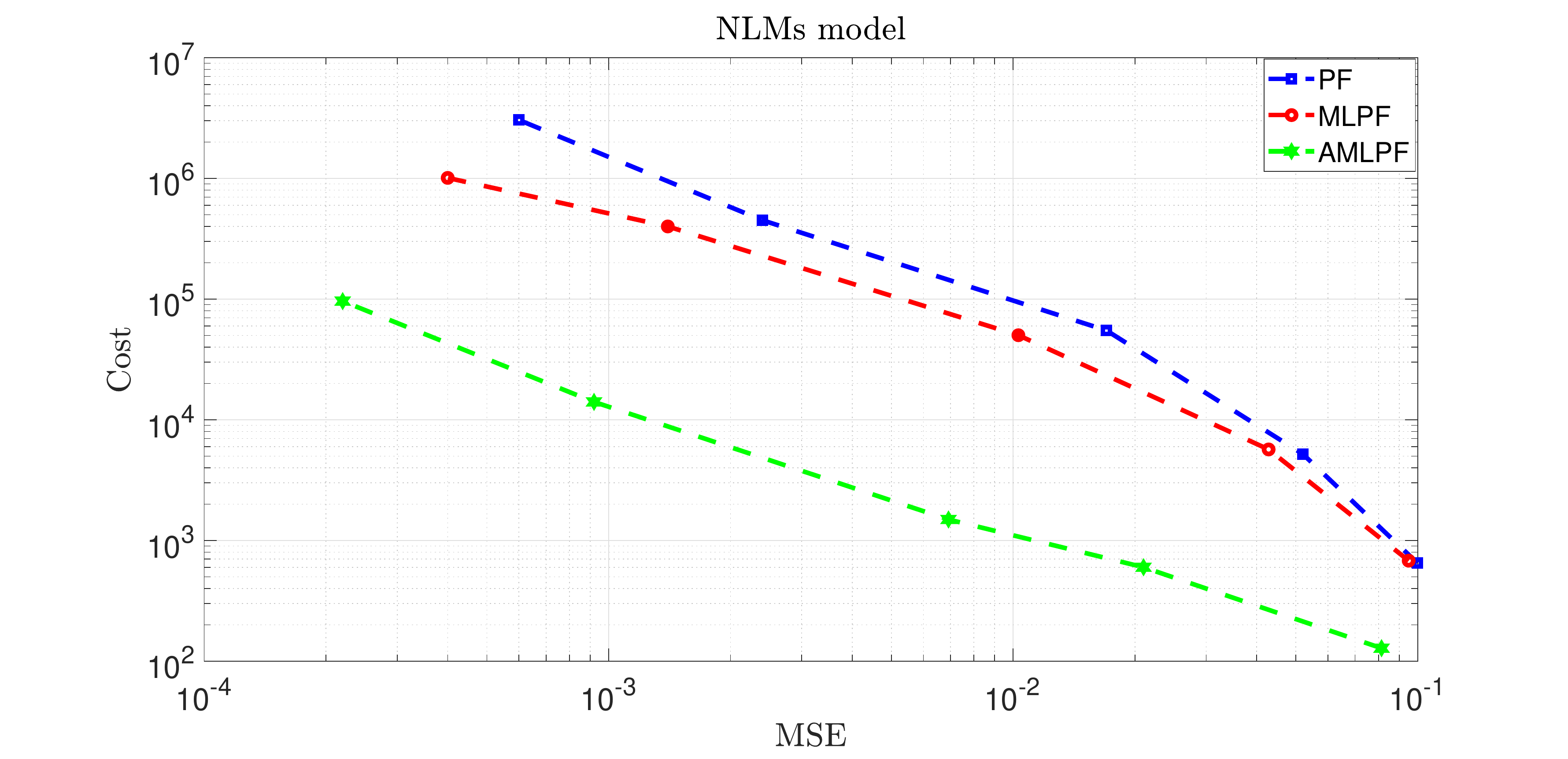}}
 \caption{Cost rates as a function of the mean squared error. The results are for the filter.}
    \label{fig:MSEvsCost}
\end{figure}

\begin{figure}[h!]
\centering
\subfigure{\includegraphics[width=10cm,height=4.5cm]{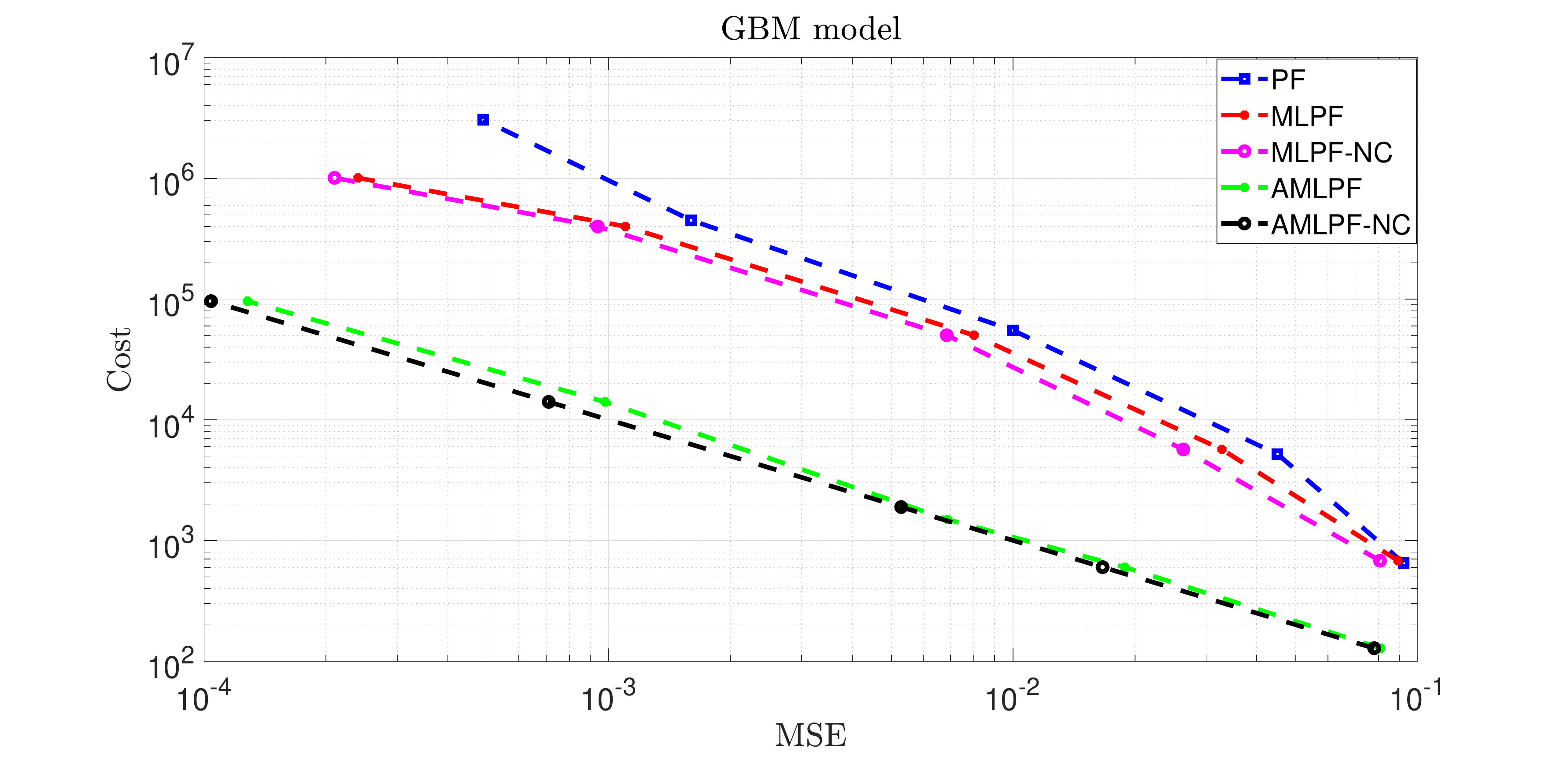}}
\subfigure{\includegraphics[width=10cm,height=4.5cm]{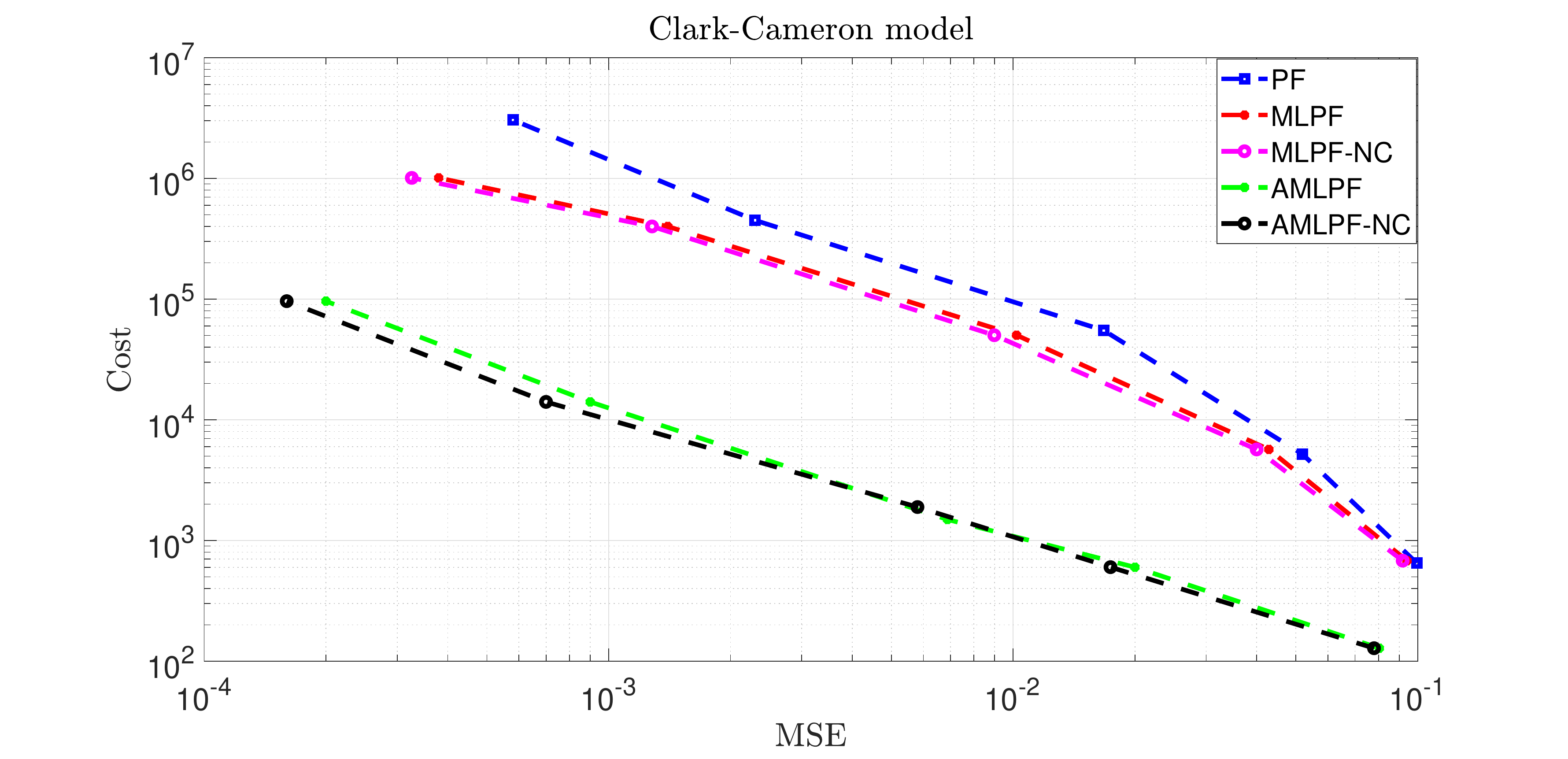}}
\subfigure{\includegraphics[width=10cm,height=4.5cm]{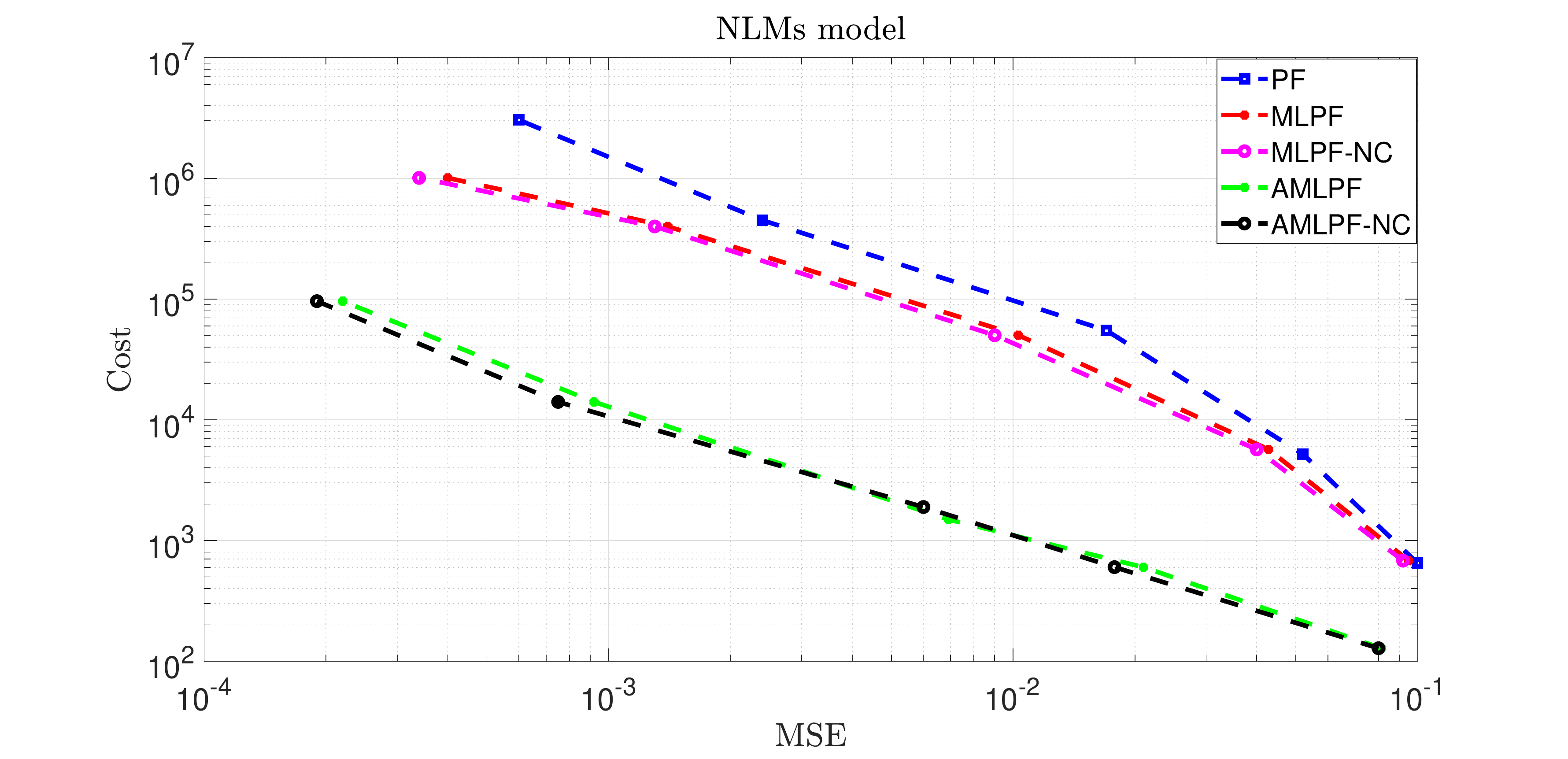}}
 \caption{Cost rates as a function of the mean squared error of our algorithms with results for the normalizing constants.}
    \label{fig:MSEvsCost-NC}
\end{figure}

\begin{table}[h!]
\begin{center}
\begin{tabular}{ c c c c c c } 
\hline \hline
Model  & PF (filter) & MLPF (NC) & MLPF (filter) & AMLPF (NC) & AMLPF (filter)  \\
\hline
\hline
\textbf{GBM} & -1.53 &-1.21  & -1.23 & -1.04  & -1.02 \\ 
 
\textbf{Clark-Cameron} & -1.55 &-1.28  & -1.26 & -1.07 & -1.05 \\
 
\textbf{NLMs} & -1.56 &-1.29  & -1.27 & -1.08 & -1.06 \\
 \hline
\end{tabular}
\caption{Estimated rates of MSE with respect to Cost. NC stands for normalizing constant.}
\label{tab:tab1}
\end{center}
\end{table}

\subsubsection*{Acknowledgements}

All authors were supported by KAUST baseline funding.

\appendix

\section{Introduction and Some Notations}\label{app:intro}

The purpose of this Appendix is to provide the necessary technical results to prove Theorem \ref{theo:main_theo}.
To understand the arguments fully, we advise that the proofs are read in order. It is possible to skip Section \ref{app:mil}
and simply refer to the proofs, but some notations or nuances could be missed. 
The structure of this appendix is to first consider some properties of the antithetic truncated Milstein scheme of \cite{ml_anti}. We present several new results which are needed directly for our subsequent proofs. We then consider the coupled particle filter as in Algorithm \ref{alg:coup_pf} in Section \ref{app:cpf}. This Section is split into three different subsections. The first of which, Section \ref{app:rates}, considers the combination of the results in Section \ref{app:mil} to the application of a coupled particle filter. These latter results then feed into the final two, which consider the convergence of $[\pi_{k}^l-\pi_{k}^{l-1}]^{N_l}(\varphi)$ in $\mathbb{L}_p$ (Section \ref{app:part_lp}) and the associated bias (Section \ref{app:part_bias}).
This the culmination of the work in this Appendix and is summarized in Theorems \ref{theo:var} and \ref{theo:bias}.
To prove our results we use two major assumptions (A\ref{ass:diff1}) and (A\ref{ass:g}). The former can be found at the start of Section \ref{app:mil} and the latter at the start of Section \ref{app:cpf}.

\subsection{Some Notations}

Let $(\mathsf{V},\mathcal{V})$ be a measurable space.
For $\varphi:\mathsf{V}\rightarrow\mathbb{R}$ we write $\mathcal{B}_b(\mathsf{V})$ as the collection of bounded measurable functions. 
Let $\varphi:\mathbb{R}^d\rightarrow\mathbb{R}$, $\textrm{Lip}(\mathbb{R}^{d})$ denotes the collection of real-valued functions that are Lipschitz w.r.t.~$\|\cdot\|$ ($\|\cdot\|$ denotes the $\mathbb{L}_2-$norm of a vector $x\in\mathbb{R}^d$). That is, $\varphi\in\textrm{Lip}(\mathbb{R}^{d})$ if there exists a $C<+\infty$ such that for any $(x,y)\in\mathbb{R}^{2d}$
$$
|\varphi(x)-\varphi(y)| \leq C\|x-y\|.
$$
For $\varphi\in\mathcal{B}_b(\mathsf{V})$, we write the supremum norm $\|\varphi\|_{\infty}=\sup_{x\in\mathsf{V}}|\varphi(x)|$.
For a measure $\mu$ on $(\mathsf{V},\mathcal{V})$
and a function $\varphi\in\mathcal{B}_b(\mathsf{V})$, the notation $\mu(\varphi)=\int_{\mathsf{V}}\varphi(x)\mu(dx)$ is used. For $A\in\mathcal{V}$, the indicator function is written as $\mathbb{I}_A(x)$.
If $K:\mathsf{V}\times\mathcal{V}\rightarrow[0,\infty)$ is a non-negative operator and $\mu$ is a measure, we use the notations
$
\mu K(dy) = \int_{\mathsf{V}}\mu(dx) K(x,dy)
$
and for $\varphi\in\mathcal{B}_b(\mathsf{V})$, 
$
K(\varphi)(x) = \int_{\mathsf{V}} \varphi(y) K(x,dy).
$
We denote (throughout) $C$ as a generic finite constant whose value may change upon each appearance and whose dependencies (on model and simulation parameters) are clear from the statements associated to them.

\section{Proofs for Antithetic Truncated Milstein Scheme}\label{app:mil}

The proofs of this section focus on the antithetic truncated Milstein discretization over a unit time (i.e.~as in Algorithm \ref{alg:milstein}). The case we consider is almost identical to that in \cite{ml_anti} except that we impose that our initial points $(x_0^l,x_0^{l-1},x_0^{l,a})$ need not be equal. This constraint is important when considering our subsequent proofs for the coupled (and multilevel) particle filter. Our objective is to prove a similar result to \cite[Theorem 4.10]{ml_anti} and to that end we make the following assumption which is stronger than  \cite[Assumption 4.1]{ml_anti}. The stronger assumption is related to the boundedness of the drift and diffusion coefficients of \eqref{eq:diff_proc}.
The reason that we require this is because it greatly simplifies our (subsequent) proofs if the constants in the below results do not depend on the initial points $(x_0^l,x_0^{l-1},x_0^{l,a})$; this would be the case otherwise. We write $\mathsf{X}_2=\mathbb{R}^{d\times d}$. $\mathbb{E}$ denotes the expectation w.r.t.~the law associated to Algorithm \ref{alg:milstein}.
The assumption is as follows.
\begin{hypA}\label{ass:diff1}
\begin{itemize}
\item{For each $(i,j)\in\{1,\dots,d\}$, $\alpha_i\in\mathcal{B}_b(\mathsf{X})$, $\beta_{ij}\in\mathcal{B}_b(\mathsf{X})$.}
\item{$\alpha\in\mathcal{C}^2(\mathsf{X},\mathsf{X})$, $\beta\in\mathcal{C}^2(\mathsf{X},\mathsf{X}_2)$.} 
\item{$\beta(x)\beta(x)^{\top}$ is uniformly positive definite.}
\item{There exists a $C<+\infty$ such that for any $(x,i,j,k,m)\in\mathsf{X}\times\{1,\dots,d\}^{4}$:
$$
\max\left\{\left|\frac{\partial \alpha_i}{\partial x_m}(x)\right|,
\left|\frac{\partial \beta_{ij}}{\partial x_m}(x)\right|,
\left|\frac{\partial h_{ijk}}{\partial x_m}(x)\right|,
\left|\frac{\partial^2 \alpha_i}{\partial x_{k}\partial x_m}(x)\right|,
\left|\frac{\partial^2 \beta_{ij}}{\partial x_{k}\partial x_m}(x)\right|
\right\} \leq C.
$$}
\end{itemize}
\end{hypA}

Our final result of this Section, Proposition \ref{prop:diff1}, is our adaptation of \cite[Theorem 4.10]{ml_anti} and is proved using simple modifications of \cite[Lemmata 4.6,4.7,4.9, Corollary 4.8]{ml_anti}. The proofs of \cite[Lemmata 4.7, Corollary 4.8]{ml_anti} need not be modified, so we proceed to prove analogues of \cite[Lemmata 4.6,4.9]{ml_anti}. 

\begin{lem}\label{lem:diff1}
Assume (A\ref{ass:diff1}). Then for any $p\in[1,\infty)$ there exists a $C<+\infty$ such that for any $l\in\mathbb{N}$
$$
\mathbb{E}\left[\max_{k\in\{0,2,\dots,\Delta_l^{-1}\}}\|X_{k\Delta_l}^{l}-X_{k\Delta_l}^{l,a}\|^p\right] \leq C\left(\Delta_l^{\tfrac{p}{2}}+\|x_0^{l}-x_0^{l-1}\|^p+\|x_0^{l,a}-x_0^{l-1}\|^p\right).
$$
\end{lem}

\begin{proof}
We have, using the $C_p-$inequality, that
$$
\mathbb{E}\left[\max_{k\in\{0,2,\dots,\Delta_l^{-1}\}}\|X_{k\Delta_l}^{l}-X_{k\Delta_l}^{l,a}\|^p\right] \leq 
$$
$$
C\left(\mathbb{E}\left[\max_{k\in\{0,2,\dots,\Delta_l^{-1}\}}\|X_{k\Delta_l}^{l}-X_{k\Delta_l}^{l-1}\|^p\right]+\mathbb{E}\left[\max_{k\in\{0,2,\dots,\Delta_l^{-1}\}}\|X_{k\Delta_l}^{l,a}-X_{k\Delta_l}^{l-1}\|^p\right]\right).
$$
Denoting the solution of \eqref{eq:diff_proc} with initial point $x$ at time $t$ as $X_t^x$, we have
$$
\mathbb{E}\left[\max_{k\in\{0,2,\dots,\Delta_l^{-1}\}}\|X_{k\Delta_l}^{l}-X_{k\Delta_l}^{l,a}\|^p\right]  \leq 
$$
\begin{eqnarray*}
& & 
C\Bigg(
\mathbb{E}\left[\max_{k\in\{0,2,\dots,\Delta_l^{-1}\}}\|X_{k\Delta_l}^{l}-X_{k\Delta_l}^{x_0^l}\|^p\right] + \mathbb{E}\left[\max_{k\in\{0,2,\dots,\Delta_l^{-1}\}}\|X_{k\Delta_l}^{x_0^l}-X_{k\Delta_l}^{x_0^{l-1}}\|^p\right] + \\ & &
\mathbb{E}\left[\max_{k\in\{0,2,\dots,\Delta_l^{-1}\}}\|X_{k\Delta_l}^{l-1}-X_{k\Delta_l}^{x_0^{l-1}}\|^p\right] + 
\mathbb{E}\left[\max_{k\in\{0,2,\dots,\Delta_l^{-1}\}}\|X_{k\Delta_l}^{l,a}-X_{k\Delta_l}^{x_0^{l,a}}\|^p\right] + \\ & & \mathbb{E}\left[\max_{k\in\{0,2,\dots,\Delta_l^{-1}\}}\|X_{k\Delta_l}^{x_0^{l,a}}-X_{k\Delta_l}^{x_0^{l-1}}\|^p\right] + 
\mathbb{E}\left[\max_{k\in\{0,2,\dots,\Delta_l^{-1}\}}\|X_{k\Delta_l}^{l-1}-X_{k\Delta_l}^{x_0^{l-1}}\|^p\right]
\Bigg).
\end{eqnarray*}
The proof is then easily concluded via the strong convergence result \cite[Lemma 4.2]{ml_anti} and standard results for diffusion processes (e.g.~via Gronwall and \cite[Corollary V.11.7]{rogers}).
\end{proof}

\begin{rem}\label{rem:diff1}
Note that the proof also establishes that: for any $p\in[1,\infty)$ there exists a $C<+\infty$ such that for any $l\in\mathbb{N}$
$$
\mathbb{E}\left[\max_{k\in\{0,2,\dots,\Delta_l^{-1}\}}\|X_{k\Delta_l}^{l}-X_{k\Delta_l}^{l-1}\|^p\right] \leq C\left(\Delta_l^{\tfrac{p}{2}}+\|x_0^{l}-x_0^{l-1}\|^p\right)
$$
and 
$$
\mathbb{E}\left[\max_{k\in\{0,2,\dots,\Delta_l^{-1}\}}\|X_{k\Delta_l}^{l,a}-X_{k\Delta_l}^{l-1}\|^p\right] \leq C\left(\Delta_l^{\tfrac{p}{2}}+\|x_0^{l,a}-x_0^{l-1}\|^p\right).
$$
\end{rem}

To state our next result, our mirror result of  \cite[Lemma 4.9]{ml_anti}, we need to introduce a significant amount of notation directly from \cite{ml_anti}. Below we set $\overline{X}_{k\Delta_l}^l=\tfrac{1}{2}(X_{k\Delta_l}^l+X_{k\Delta_l}^{l,a})$, $k\in\{0,\dots,\Delta_l^{-1}\}$. For $(i,k)\in\{1,\dots,d\}\times\{2,4,\dots,\Delta_l^{-1}\}$ we define
\begin{eqnarray*}
R_{i,k}^{(1)} & = & \left\{\frac{1}{2}\left(\alpha_i(X_{k\Delta_l}^l)+\alpha_i(X_{k\Delta_l}^{l,a})\right)-\alpha_i(\overline{X}_{k\Delta_l}^l)\right\}\Delta_{l-1} \\
M_{i,k}^{(1)} & = & \sum_{j=1}^d\left\{\frac{1}{2}\left(\beta_{ij}(X_{k\Delta_l}^l)+\beta_{ij}(X_{k\Delta_l}^{l,a})\right)-
\beta_{ij}(\overline{X}_{k\Delta_l}^l)
\right\}[W_{j,(k+2)\Delta_l}-W_{j,k\Delta_l}] \\
M_{i,k}^{(2)} & = & \sum_{(j,m)\in\{1,\dots,d\}^2}
\left\{\frac{1}{2}\left(h_{ijm}(X_{k\Delta_l}^l)+h_{ijm}(X_{k\Delta_l}^{l,a})\right)-
h_{ijm}(\overline{X}_{k\Delta_l}^l)
\right\}\Big([W_{j,(k+2)\Delta_l}-W_{j,k\Delta_l}]\times \\ & &[W_{m,(k+2)\Delta_l}-W_{m,k\Delta_l}]-\Delta_{l-1}\Big)\\
M_{i,k}^{(3)} & = & \sum_{(j,m)\in\{1,\dots,d\}^2}\frac{1}{2}\left(h_{ijm}(X_{k\Delta_l}^l)-h_{ijm}(X_{k\Delta_l}^{l,a})\right)
\Big([W_{j,(k+1)\Delta_l}-W_{j,k\Delta_l}][W_{m,(k+2)\Delta_l}-W_{m,(k+1)\Delta_l}] - \\ & &
[W_{m,(k+1)\Delta_l}-W_{m,k\Delta_l}][W_{j,(k+2)\Delta_l}-W_{j,(k+1)\Delta_l}]
\Big)
\end{eqnarray*}
\begin{eqnarray*}
R_{i,k}^l & = & \sum_{j=1}^d\frac{\partial\alpha_i}{\partial x_j}(X_{k\Delta_l}^l)\Bigg(\alpha_{j}(X_{k\Delta_l}^l)\Delta_l
+ \sum_{(m,n)\in\{1,\dots,d\}^2} h_{imn}(X_{k\Delta_l}^l)\Big([W_{m,(k+1)\Delta_l}-W_{m,k\Delta_l}][W_{n,(k+1)\Delta_l}-\\ & & W_{n,k\Delta_l}]- \Delta_l\Big)\Bigg)\Delta_l + \frac{1}{2}\sum_{(j,m)\in\{1,\dots,d\}^2}\frac{\partial^2\alpha_i}{\partial x_j\partial x_m}(\xi_1^l)[X_{j,(k+1)\Delta_l}^l-X_{j,k\Delta_l}^l][X_{m,(k+1)\Delta_l}^l-X_{m,k\Delta_l}^l]\Delta_l \\
M_{i,k}^{(1,l)} & = & \sum_{(j,m)\in\{1,\dots,d\}^2}\frac{\partial\alpha_i}{\partial x_j}(X_{k\Delta_l}^l)\beta_{jm}(X_{k\Delta_l}^l)[W_{m,(k+1)\Delta_l}-W_{m,k\Delta_l}]\Delta_l
\end{eqnarray*}
\begin{eqnarray*}
M_{i,k}^{(2,l)} & = & \sum_{(j,m)\in\{1,\dots,d\}^2} \frac{\partial\beta_{ij}}{\partial x_m}(X_{k\Delta_l}^l)
\Bigg(\alpha_m(X_{k\Delta_l}^l)\Delta_l +
\sum_{(n,p)\in\{1,\dots,d\}^2}h_{mnp}(X_{k\Delta_l}^l)[
[W_{n,(k+1)\Delta_l}-W_{n,k\Delta_l}]\times \\ & & [W_{p,(k+1)\Delta_l}-W_{p,k\Delta_l}]
-\Delta_l]
\Bigg)[W_{j,(k+2)\Delta_l}-W_{j,(k+1)\Delta_l}]\\
M_{i,k}^{(3,l)} & = & \sum_{(j,m)\in\{1,\dots,d\}^2}\left\{h_{ijm}(X_{(k+1)\Delta_l}^l)-h_{ijm}(X_{k\Delta_l}^l)\right\}
\Big(
[W_{j,(k+2)\Delta_l}-W_{j,(k+1)\Delta_l}][W_{m,(k+2)\Delta_l}-\\ & & W_{m,(k+1)\Delta_l}]
-\Delta_l
\Big)
\end{eqnarray*}
where in $R_{i,k}^l$ $\xi_1^l$ is some point which lies on the line between $X_{k\Delta_l}^l$ and $X_{(k+1)\Delta_l}^l$.
In the case of $R_{i,k}^l$ and $M_{i,k}^{(j,l)}$, $j\in\{1,2,3\}$ one can substitute $X_{k\Delta_l}^l,\xi_1^{l},X_{(k+1)\Delta_l}^l$ for $X_{k\Delta_l}^{l,a},\xi_1^{l,a},X_{(k+1)\Delta_l}^{l,a}$ (where $\xi_1^{l,a}$ is some point which lies on the line between $X_{k\Delta_l}^{l,a}$ and $X_{(k+1)\Delta_l}^{l,a}$) and when we do so, we use the notations $R_{i,k}^{l,a}$ and $M_{i,k}^{(j,l,a)}$, $j\in\{1,2,3\}$. Finally we set
\begin{eqnarray*}
M_k & = & \sum_{j=1}^3 M_k^{(j)} + \frac{1}{2}\sum_{j=1}^3\{M_k^{(j,l)} + M_k^{(j,l,a)}\} \\
R_k & = & R_k^{(1)} + \frac{1}{2}\{R_{i,k}^l+R_{i,k}^{l,a}\}.
\end{eqnarray*}
We are now in the position to give our analogue of \cite[Lemma 4.9]{ml_anti}.

\begin{lem}\label{lem:diff2}
Assume (A\ref{ass:diff1}). Then:
\begin{itemize}
\item{For any $p\in[1,\infty)$ there exists a $C<+\infty$ such that for any $l\in\mathbb{N}$
$$
\max_{k\in\{2,4,\dots,\Delta_l\}}\mathbb{E}[|R_k|^p] \leq C\Delta_l^p\left(\Delta_l^p + \|x_0^{l}-x_0^{l-1}\|^{2p}+\|x_0^{l,a}-x_0^{l-1}\|^{2p}\right).
$$}
\item{For any $p\in[1,\infty)$ there exists a $C<+\infty$ such that for any $l\in\mathbb{N}$
$$
\max_{k\in\{2,4,\dots,\Delta_l\}}\mathbb{E}[|M_k|^p] \leq 
$$
$$
C\left(\Delta_l^{\tfrac{3p}{2}}+
\Delta_l^{\tfrac{p}{2}}\{\|x_0^{l}-x_0^{l-1}\|^{2p}+\|x_0^{l,a}-x_0^{l-1}\|^{2p}\} +
\Delta_l^{p}\{\|x_0^{l}-x_0^{l-1}\|^{p}+\|x_0^{l,a}-x_0^{l-1}\|^{p}\}
\right).
$$}
\end{itemize}
\end{lem}

\begin{proof}
The proof of this result essentially follows \cite{ml_anti} and controlling the terms (in $\mathbb{L}_p$). The expressions $\frac{1}{2}\sum_{j=1}^3\{M_k^{(j,l)} + M_k^{(j,l,a)}\}$ and $\frac{1}{2}\{R_{i,k}^l+R_{i,k}^{l,a}\}$ can be dealt with exactly as in 
\cite[Lemma 4.7]{ml_anti} so we need only consider the terms $\sum_{j=1}^3 M_k^{(j)}$ and $R_k^{(1)}$. It will suffice to control in $\mathbb{L}_p$ any of the $d$ co-ordinates of the afore mentioned vectors, which is what we do below.

Beginning with $R_{i,k}^{(1)}$, using the second order Taylor expansion in \cite{ml_anti}, one has
$$
R_{i,k}^{(1)} = \frac{1}{8}\sum_{(j,m)\in\{1,\dots,d\}^2}\left(\frac{\partial^2\alpha_i}{\partial x_j\partial x_m}(\xi_1^l)+
\frac{\partial^2\alpha_i}{\partial x_j\partial x_m}(\xi_2^l)\right)(X_{j,k\Delta_l}^l - X_{j,k\Delta_l}^{l,a})(X_{m,k\Delta_l}^l - X_{m,k\Delta_l}^{l,a})\Delta_l
$$
where $\xi_1^l$ is some point between $\overline{X}^l_{k\Delta_l}$ and $X^l_{k\Delta_l}$ and
$\xi_2^l$ is some point between $\overline{X}^l_{k\Delta_l}$ and $X^{l,a}_{k\Delta_l}$. Then it follows easily that 
$$
\mathbb{E}[|R_{i,k}^{(1)}|^p] \leq C\Delta_l^p\mathbb{E}[\|X_{k\Delta_l}^l -X_{k\Delta_l}^{l,a}\|^{2p}].
$$
Application of Lemma \ref{lem:diff1} yields that 
$$
\mathbb{E}[|R_{i,k}^{(1)}|^p] \leq C\Delta_l^p\left(\Delta_l^p + \|x_0^{l}-x_0^{l-1}\|^{2p}+\|x_0^{l,a}-x_0^{l-1}\|^{2p}\right).
$$
which is the desired result.

For $M_{i,k}^{(1)}$ again one has
\begin{eqnarray*}
M_{i,k}^{(1)} & = & \frac{1}{16}\sum_{(j,m,n)\in\{1,\dots,d\}^3}\left\{
\frac{\partial^2\beta_{ij}}{\partial x_m\partial x_n}(\xi_3^l) + 
\frac{\partial^2\beta_{ij}}{\partial x_m\partial x_n}(\xi_4^l)
\right\}
(X_{m,k\Delta_l}^l - X_{m,k\Delta_l}^{l,a})(X_{n,k\Delta_l}^l - X_{n,k\Delta_l}^{l,a})\times \\ & & [W_{j,(k+2)\Delta_l}-W_{j,k\Delta_l}]
\end{eqnarray*}
where $\xi_3^l$ is some point between $\overline{X}^l_{k\Delta_l}$ and $X^l_{k\Delta_l}$ and
$\xi_4^l$ is some point between $\overline{X}^l_{k\Delta_l}$ and $X^{l,a}_{k\Delta_l}$. Then using the independence of the Brownian increment (with the random variables $(X_{m,k\Delta_l}^l - X_{m,k\Delta_l}^{l,a})(X_{n,k\Delta_l}^l - X_{n,k\Delta_l}^{l,a})$) and the same approach as above, one has
$$
\mathbb{E}[|M_{i,k}^{(1)}|^p] \leq C\Delta_l^{p/2}\mathbb{E}[\|X_{k\Delta_l}^l -X_{k\Delta_l}^{l,a}\|^{2p}]
$$
and applying Lemma \ref{lem:diff1} one has
$$
\mathbb{E}[|M_{i,k}^{(1)}|^p] \leq C\Delta_l^{p/2}\left(\Delta_l^p + \|x_0^{l}-x_0^{l-1}\|^{2p}+\|x_0^{l,a}-x_0^{l-1}\|^{2p}\right).
$$

For $M_{i,k}^{(2)}$ one has
\begin{eqnarray*}
M_{i,k}^{(2)} & = & \frac{1}{4}\sum_{(j,m,n)\in\{1,\dots,d\}^3}\left\{
\frac{\partial h_{ijm}}{\partial x_n}(\xi_5^l) +
\frac{\partial h_{ijm}}{\partial x_n}(\xi_6^l)
\right\}(X_{n,k\Delta_l}^l - X_{n,k\Delta_l}^{l,a})\times \\ & &
\Big([W_{j,(k+2)\Delta_l}-W_{j,k\Delta_l}][W_{m,(k+2)\Delta_l}-W_{m,k\Delta_l}]-\Delta_{l-1}\Big)
\end{eqnarray*}
where $\xi_5^l$ is some point between $\overline{X}^l_{k\Delta_l}$ and $X^l_{k\Delta_l}$ and
$\xi_6^l$ is some point between $\overline{X}^l_{k\Delta_l}$ and $X^{l,a}_{k\Delta_l}$. 
As the Brownian increments are independent of $(X_{n,k\Delta_l}^l - X_{n,k\Delta_l}^{l,a})$ and the dimensions are
also independent of each other, one yields
$$
\mathbb{E}[|M_{i,k}^{(2)}|^p] \leq C\Delta_l^{p}\mathbb{E}[\|X_{k\Delta_l}^l -X_{k\Delta_l}^{l,a}\|^{p}].
$$
Lemma \ref{lem:diff1} gives
$$
\mathbb{E}[|M_{i,k}^{(2)}|^p] \leq C\Delta_l^{p}\left(\Delta_l^{p/2} + \|x_0^{l}-x_0^{l-1}\|^{p}+\|x_0^{l,a}-x_0^{l-1}\|^{p}\right).
$$

For $M_{i,k}^{(3)}$ 
\begin{eqnarray*}
M_{i,k}^{(3)} & = & \frac{1}{4}\sum_{(j,m,n)\in\{1,\dots,d\}^3}\left\{
\frac{\partial h_{ijm}}{\partial x_n}(\xi_7^l) + 
\frac{\partial h_{ijm}}{\partial x_n}(\xi_8^l)
\right\}
(X_{k\Delta_l}^l-X_{k\Delta_l}^{l,a})
\Big([W_{j,(k+1)\Delta_l}-W_{j,k\Delta_l}]\times \\ & & [W_{m,(k+2)\Delta_l}-W_{m,(k+1)\Delta_l}] - 
[W_{m,(k+1)\Delta_l}-W_{m,k\Delta_l}][W_{j,(k+2)\Delta_l}-W_{j,(k+1)\Delta_l}]
\Big)
\end{eqnarray*}
where $\xi_7^l$ is some point between $\overline{X}^l_{k\Delta_l}$ and $X^l_{k\Delta_l}$ and
$\xi_8^l$ is some point between $\overline{X}^l_{k\Delta_l}$ and $X^{l,a}_{k\Delta_l}$.  Using essentially the same approach as for $M_{i,k}^{(2)}$ one obtains
$$
\mathbb{E}[|M_{i,k}^{(3)}|^p] \leq C\Delta_l^{p}\left(\Delta_l^{p/2} + \|x_0^{l}-x_0^{l-1}\|^{p}+\|x_0^{l,a}-x_0^{l-1}\|^{p}\right).
$$
which concludes the proof.
\end{proof}

\begin{prop}\label{prop:diff1}
Assume (A\ref{ass:diff1}). Then for any $p\in[1,\infty)$ there exists a $C<+\infty$ such that for any $l\in\mathbb{N}$
$$
\mathbb{E}\left[\max_{k\in\{0,2,\dots,\Delta_l^{-1}\}}\|\overline{X}_{k\Delta_l}^{l}-X_{k\Delta_l}^{l-1}\|^p\right]
\leq
$$
$$
C\left(
\Delta_l^p + \{\|x_0^{l}-x_0^{l-1}\|^{2p}+\|x_0^{l,a}-x_0^{l-1}\|^{2p}\} + \Delta_l^{\tfrac{p}{2}}
\{\|x_0^{l}-x_0^{l-1}\|^{p}+\|x_0^{l,a}-x_0^{l-1}\|^{p}\}
\right).
$$
\end{prop}

\begin{proof}
The proof is identical to \cite[Theorem 4.10]{ml_anti}, except that one uses Lemma \ref{lem:diff2} above, instead of \cite[Lemma 4.9]{ml_anti} and is hence omitted.
\end{proof}

\section{Proofs for Coupled Particle Filter}\label{app:cpf}

To continue with our proofs, we introduce an additional assumption. We drop dependence upon the data in $g(x_k,y_k)$
and simply write $g_k(x_k)$.

\begin{hypA}\label{ass:g}
\begin{itemize}
\item{For each $k\in\mathbb{N}$, $g_k\in\mathcal{B}_b(\mathsf{X})\cap\mathcal{C}^2(\mathsf{X},\mathbb{R})$.}
\item{For each $k\in\mathbb{N}$ there exists a $0<C<+\infty$ such that for any $x\in\mathsf{X}$ $g_k(x)\geq C$.}
\item{For each $k\in\mathbb{N}$ there exists a $0<C<+\infty$ such that for any $(x,j,m)\in\mathsf{X}\times\{1,\dots,d\}^2$:
$$
\max\left\{\Big|\frac{\partial g_{k}}{\partial x_j}(x)\Big|,\Big|\frac{\partial^2 g_{k}}{\partial x_j\partial x_m}(x)\Big|\right\} \leq C.
$$
}
\end{itemize}
\end{hypA}
In this section $\mathbb{E}$ denotes the expectation w.r.t.~the law that generates the AMLPF.

\subsection{Rate Proofs}\label{app:rates}

Our analysis will apply any coupling used in Algorithm \ref{alg:max_coup} step 2.~bullet 2. At any time point, $k$ of Algorithm  \ref{alg:coup_pf} we will denote the resampled index of particle $i\in\{1,\dots,N_l\}$ as $I_k^{i,l}$ (level $l$), $I_k^{i,l-1}$ (level $l-1$) and $I_k^{i,l,a}$ (level $l$ antithetic). Now let $I_k^l(i)=I_k^{i,l}$, $I_k^{l-1}(i)=I_k^{i,l-1}$, $I_k^{l,a}(i)=I_k^{i,l,a}$ and define $\mathsf{S}_k^l$ the collection of indices that choose the same ancestor at each resampling step, i.e.
\begin{eqnarray*}
\mathsf{S}_k^l & = & \{i\in\{1,\dots,N_l\}:I_k^l(i)=I_k^{l-1}(i)=I_k^{l,a}(i),I_{k-1}^l\circ I_k^l(i)=I_{k-1}^{l-1}\circ I_k^{l-1}(i)=I_{k-1}^{l,a}\circ I_k^{l,a}(i),\dots,\\ & & I_{1}^l\circ I_2^l\circ\cdots\circ I_k^l(i) = I_{1}^{l-1}\circ I_2^{l-1}\circ\cdots\circ I_k^{l-1}(i)=I_{1}^{l,a}\circ I_2^{l,a}\circ\cdots\circ I_k^{l,a}(i)\}.
\end{eqnarray*}
We use the convention that $\mathsf{S}_0^l=\{1,\dots,N_l\}$.
Denote the $\sigma-$field generated by the simulated samples, resampled samples and resampled indices up-to time
$k$ as $\hat{\mathcal{F}}_k^{l}$ and the $\sigma-$field  which does the same, except excluding the resampled samples and indices as $\mathcal{F}_k^{l}$. 

\begin{lem}\label{lem:cpf1}
Assume (A\ref{ass:diff1}-\ref{ass:g}). Then for any $(p,k)\in[1,\infty)\times\mathbb{N}$ there exists a $C<+\infty$ such that for any $(l,N_l)\in\mathbb{N}^2$
$$
\mathbb{E}\left[\frac{1}{N_l}\sum_{i\in\mathsf{S}_{k-1}^l}\left\{\|X_k^{i,l}-X_k^{i,l-1}\|^p + 
\|X_k^{i,l,a}-X_k^{i,l-1}\|^p + 
\|X_k^{i,l}-X_k^{i,l,a}\|^p
\right\}\right] \leq C\Delta_l^{p/2}.
$$
\end{lem}

\begin{proof}
The case $k=1$ follows from the work in \cite{ml_anti} (the case $p=[1,2)$ can be adapted from that paper), so we assume $k\geq 2$.
We have by conditioning on $\hat{\mathcal{F}}_{k-1}^{l}$ and applying Lemma \ref{lem:diff1} (see also Remark \ref{rem:diff1})
$$
\mathbb{E}\left[\frac{1}{N_l}\sum_{i\in\mathsf{S}_{k-1}^l}\left\{\|X_k^{i,l}-X_k^{i,l-1}\|^p + 
\|X_k^{i,l,a}-X_k^{i,l-1}\|^p + 
\|X_k^{i,l}-X_k^{i,l,a}\|^p
\right\}\right]\leq 
$$
$$
C\left(\Delta_l^{p/2} + \mathbb{E}\left[\frac{1}{N_l}\sum_{i\in\mathsf{S}_{k-1}^l}\left\{\|\hat{X}_{k-1}^{i,l}-\hat{X}_{k-1}^{i,l-1}\|^p + 
\|\hat{X}_{k-1}^{i,l,a}-\hat{X}_{k-1}^{i,l-1}\|^p + 
\|\hat{X}_{k-1}^{i,l}-\hat{X}_{k-1}^{i,l,a}\|^p
\right\}\right]  \right).
$$
One can exchangeably write the expectation on the R.H.S.~as
$$
\mathbb{E}\left[\frac{1}{N_l}\sum_{i\in\mathsf{S}_{k-1}^l}\left\{\|\hat{X}_{k-1}^{I_{k-1}^{i,l},l}-\hat{X}_{k-1}^{I_{k-1}^{i,l-1},l-1}\|^p + 
\|\hat{X}_{k-1}^{I_{k-1}^{i,l,a},l,a}-\hat{X}_{k-1}^{I_{k-1}^{i,l-1},l-1}\|^p + 
\|\hat{X}_{k-1}^{I_{k-1}^{i,l},l}-\hat{X}_{k-1}^{I_{k-1}^{i,l,a},l,a}\|^p
\right\}\right]. 
$$
The proof from here is then essentially identical (up-to the fact that one has three indices instead of two) to that of \cite[Lemma D.3]{mlpf} and is hence omitted.
\end{proof}

\begin{lem}\label{lem:cpf2}
Assume (A\ref{ass:diff1}-\ref{ass:g}). Then for any $(p,k)\in[1,\infty)\times\mathbb{N}$ there exists a $C<+\infty$ such that for any $(l,N_l)\in\mathbb{N}^2$
$$
\mathbb{E}\left[\frac{1}{N_l}\sum_{i\in\mathsf{S}_{k-1}^l}\|\overline{X}_k^{i,l}-X_k^{i,l-1}\|^p\right] \leq C\Delta_l^{p}.
$$
\end{lem}

\begin{proof}
Following the start of the proof of Lemma \ref{lem:cpf1}, except using Proposition \ref{prop:diff1} instead of Lemma \ref{lem:diff1} one can deduce that 
$$
\mathbb{E}\left[\frac{1}{N_l}\sum_{i\in\mathsf{S}_{k-1}^l}\|\overline{X}_k^{i,l}-X_k^{i,l-1}\|^p\right] \leq 
$$
$$
C\Bigg(
\Delta_l^p + \mathbb{E}\left[\frac{1}{N_l}\sum_{i\in\mathsf{S}_{k-1}^l}
\Big\{\|\hat{X}_{k-1}^{I_{k-1}^{i,l},l}-\hat{X}_{k-1}^{I_{k-1}^{i,l-1},l-1}\|^{2p} + 
\|\hat{X}_{k-1}^{I_{k-1}^{i,l,a},l,a}-\hat{X}_{k-1}^{I_{k-1}^{i,l-1},l-1}\|^{2p}\Big\}\right] +
$$
$$
\Delta_l^{p/2}
\mathbb{E}\left[\frac{1}{N_l}\sum_{i\in\mathsf{S}_{k-1}^l}
\Big\{\|\hat{X}_{k-1}^{I_{k-1}^{i,l},l}-\hat{X}_{k-1}^{I_{k-1}^{i,l-1},l-1}\|^{p} + 
\|\hat{X}_{k-1}^{I_{k-1}^{i,l,a},l,a}-\hat{X}_{k-1}^{I_{k-1}^{i,l-1},l-1}\|^{p}\Big\}\right]  
\Bigg).
$$
From here, one can follow the calculations in \cite[Lemma D.3]{mlpf} to deduce that
$$
\mathbb{E}\left[\frac{1}{N_l}\sum_{i\in\mathsf{S}_{k-1}^l}\|\overline{X}_k^{i,l}-X_k^{i,l-1}\|^p\right] \leq 
$$
$$
C\Bigg(\Delta_l^p + \mathbb{E}\left[\frac{1}{N_l}\sum_{i\in\mathsf{S}_{k-2}^l}\left\{\|X_{k-1}^{i,l}-X_{k-1}^{i,l-1}\|^{2p} + 
\|X_{k-1}^{i,l,a}-X_{k-1}^{i,l-1}\|^{2p} + 
\|X_{k-1}^{i,l}-X_{k-1}^{i,l,a}\|^{2p}
\right\}\right] +
$$
$$
\Delta_l^{p/2}\mathbb{E}\left[\frac{1}{N_l}\sum_{i\in\mathsf{S}_{k-2}^l}\left\{\|X_{k-1}^{i,l}-X_{k-1}^{i,l-1}\|^{p} + 
\|X_{k-1}^{i,l,a}-X_{k-1}^{i,l-1}\|^{p} + 
\|X_{k-1}^{i,l}-X_{k-1}^{i,l,a}\|^{p}
\right\}\right]
\Bigg).
$$
Application of Lemma \ref{lem:cpf1} concludes the result.
\end{proof}

\begin{lem}\label{lem:cpf_weight}
Assume (A\ref{ass:diff1}-\ref{ass:g}). Then for any $k\in\mathbb{N}$ there exists a $C<+\infty$ such that for any $(l,N_l)\in\mathbb{N}^2$
$$
\mathbb{E}\left[\mathbb{I}_{\mathsf{S}_{k-1}^l}(i)\left|\left\{\frac{g_{k}(X_k^{i,l})}{\sum_{j=1}^{N_l}g_{k}(X_k^{j,l})}-
\frac{g_{k}(X_k^{i,l,a})}{\sum_{j=1}^{N_l}g_{k}(X_k^{j,l,a})}\right\}
\right|\right] \leq C\left(\frac{\Delta_l^{1/2}}{N_l}+ 
\frac{1}{N_l}\left\{
1-\mathbb{E}\left[\frac{\textrm{\emph{Card}}(\mathsf{S}_{k-1}^l)}{N_l}\right]
\right\}
\right).
$$
\end{lem}

\begin{proof}
We have the elementary decomposition
$$
\frac{g_{k}(X_k^{i,l})}{\sum_{j=1}^{N_l}g_{k}(X_k^{j,l})}-
\frac{g_{k}(X_k^{i,l,a})}{\sum_{j=1}^{N_l}g_{k}(X_k^{j,l,a})} = \frac{g_{k}(X_k^{i,l})-g_{k}(X_k^{i,l,a})}{\sum_{j=1}^{N_l}g_{k}(X_k^{j,l})} + g_{k}(X_k^{i,l,a})\left(
\frac{\sum_{j=1}^{N_l}g_{k}(X_k^{j,l,a})-\sum_{j=1}^{N_l}g_{k}(X_k^{j,l})}{\sum_{j=1}^{N_l}g_{k}(X_k^{j,l,a})\sum_{j=1}^{N_l}g_{k}(X_k^{j,l})}
\right).
$$
So application of the triangular-inequality along with the uniform in $x$ lower and upper-bounds on $g_k$ gives
$$
\mathbb{E}\left[\mathbb{I}_{\mathsf{S}_{k-1}^l}(i)\left|\left\{\frac{g_{k}(X_k^{i,l})}{\sum_{j=1}^{N_l}g_{k}(X_k^{j,l})}-
\frac{g_{k}(X_k^{i,l,a})}{\sum_{j=1}^{N_l}g_{k}(X_k^{j,l,a})}\right\}
\right|\right] \leq
$$
$$
\frac{C}{N_l}\Bigg(\mathbb{E}[\mathbb{I}_{\mathsf{S}_{k-1}^l}(i)|g_{k}(X_k^{i,l})-g_{k}(X_k^{i,l,a})|] +
\frac{1}{N_l}\mathbb{E}\left[\sum_{j\in \mathsf{S}_{k-1}^l}|g_{k}(X_k^{j,l})-g_{k}(X_k^{j,l,a})|\right] + 
1-\mathbb{E}\left[\frac{\textrm{Card}(\mathsf{S}_{k-1}^l)}{N_l}\right]
\Bigg).
$$
Using the Lipschitz property of $g_k$ along with Lemma \ref{lem:cpf1} allows one to conclude the result.
\end{proof}

\begin{rem}\label{rem:cpf_weight} 
Using a similar approach to the above proof, one can establish the following result. Assume (A\ref{ass:diff1}-\ref{ass:g}). Then for any $k\in\mathbb{N}$ there exists a $C<+\infty$ such that for any $(l,N_l)\in\mathbb{N}^2$
$$
\mathbb{E}\left[\mathbb{I}_{\mathsf{S}_{k-1}^l}(i)\left|\left\{\frac{g_{k}(X_k^{i,l})}{\sum_{j=1}^{N_l}g_{k}(X_k^{j,l})}-
\frac{g_{k}(X_k^{i,l,a})}{\sum_{j=1}^{N_l}g_{k}(X_k^{j,l,a})}\right\}
\right|^2\right] \leq C\left(\frac{\Delta_l^{1/2}}{N_l^2}+ 
\frac{1}{N_l^2}\left\{
1-\mathbb{E}\left[\frac{\textrm{\emph{Card}}(\mathsf{S}_{k-1}^l)}{N_l}\right]
\right\}
\right).
$$
\end{rem}

The following is a useful Lemma that we will need below and whose proof consists of tedious algebra and is hence omitted.

\begin{lem}\label{lem:tech_lem_rat}
Let $(G^l,f^l,G^a,f^a,G^{l-1},f^{l-1})\in\mathbb{R}^6$ with $(f^l,f^a,f^{l-1})$ non-zero then
\begin{eqnarray*}
\frac{\tfrac{1}{2}G^l}{f^l} + \frac{\tfrac{1}{2}G^a}{f^a} - \frac{G^{l-1}}{f^{l-1}} & = & \frac{1}{f^{l-1}}\Big(\tfrac{1}{2}G^l+\tfrac{1}{2}G^a-G^{l-1}\Big) + \frac{1}{f^{a}f^{l}f^{l-1}}\Big(
\tfrac{1}{2}(G^l-G^a)(f^{l-1}-f^l)f^a + \\ & &
\tfrac{1}{2}G^a\Big[
(f^a-f^l)(f^{l-1}-f^l) - 2f^l\Big\{
\tfrac{1}{2}f^l + \tfrac{1}{2}f^a - f^{l-1}
\Big\}
\Big]
\Big).
\end{eqnarray*}
\end{lem}

\begin{lem}\label{lem:cpf_weight1}
Assume (A\ref{ass:diff1}-\ref{ass:g}). Then for any $k\in\mathbb{N}$ there exists a $C<+\infty$ such that for any $(l,N_l)\in\mathbb{N}^2$
$$
\mathbb{E}\left[\mathbb{I}_{\mathsf{S}_{k-1}^l}(i)\left|\left\{\frac{\tfrac{1}{2}g_{k}(X_k^{i,l})}{\sum_{j=1}^{N_l}g_{k}(X_k^{j,l})}
+
\frac{\tfrac{1}{2}g_{k}(X_k^{i,l,a})}{\sum_{j=1}^{N_l}g_{k}(X_k^{j,l,a})}
-
\frac{g_{k}(X_k^{i,l-1})}{\sum_{j=1}^{N_l}g_{k}(X_k^{j,l-1})}\right\}
\right|\right] \leq 
$$
$$
C\left(\frac{\Delta_l}{N_l}+ 
\frac{1}{N_l}\left\{
1-\mathbb{E}\left[\frac{\textrm{\emph{Card}}(\mathsf{S}_{k-1}^l)}{N_l}\right]\right\}\right).
$$
\end{lem}

\begin{proof}
One can apply Lemma \ref{lem:tech_lem_rat} to yield that 
$$
\frac{\tfrac{1}{2}g_{k}(X_k^{i,l})}{\sum_{j=1}^{N_l}g_{k}(X_k^{j,l})}
+
\frac{\tfrac{1}{2}g_{k}(X_k^{i,l,a})}{\sum_{j=1}^{N_l}g_{k}(X_k^{j,l,a})}
-
\frac{g_{k}(X_k^{i,l-1})}{\sum_{j=1}^{N_l}g_{k}(X_k^{j,l-1})} = \sum_{j=1}^4 T_j
$$
where
\begin{eqnarray*}
T_1 & = & \frac{1}{\sum_{j=1}^{N_l}g_{k}(X_k^{j,l-1})}\left\{\tfrac{1}{2}g_{k}(X_k^{i,l})+\tfrac{1}{2}g_{k}(X_k^{i,l,a})-
g_{k}(X_k^{i,l-1})
\right\} \\
T_2 & = & \frac{1}{2[\sum_{j=1}^{N_l}g_{k}(X_k^{j,l})]
[\sum_{j=1}^{N_l}g_{k}(X_k^{j,l-1})]
}\left(g_{k}(X_k^{i,l})-g_{k}(X_k^{i,l,a})\right)
\left(\sum_{j=1}^{N_l}[g_{k}(X_k^{j,l-1})-g_{k}(X_k^{j,l})]\right) \\
T_3 & = & \frac{g_{k}(X_k^{i,l,a})}{2[\sum_{j=1}^{N_l}g_{k}(X_k^{j,l,a})][\sum_{j=1}^{N_l}g_{k}(X_k^{j,l})]
[\sum_{j=1}^{N_l}g_{k}(X_k^{j,l-1})]
}
\left(\sum_{j=1}^{N_l}[g_{k}(X_k^{j,l,a})-g_{k}(X_k^{j,l})]\right)\times \\ & &
\left(\sum_{j=1}^{N_l}[g_{k}(X_k^{j,l-1})-g_{k}(X_k^{j,l})]\right)\\
T_4 & = & -\frac{g_{k}(X_k^{i,l,a})}{[\sum_{j=1}^{N_l}g_{k}(X_k^{j,l,a})]
[\sum_{j=1}^{N_l}g_{k}(X_k^{j,l-1})]
}\left\{
\sum_{j=1}^{N_l}[\tfrac{1}{2}g_{k}(X_k^{j,l}) + \tfrac{1}{2}g_{k}(X_k^{j,l,a}) -
g_{k}(X_k^{j,l-1})
].
\right\}
\end{eqnarray*}
The terms $T_1$ (resp.~$T_2$) and $T_4$ (resp.~$T_3$) can be dealt with in a similar manner, so we only consider $T_1$ (resp.~$T_2$).

For the case of $T_1$ we have the upper-bound:
$$
\mathbb{E}\left[\mathbb{I}_{\mathsf{S}_{k-1}^l}(i)T_1\right] \leq \frac{C}{N_l}\mathbb{E}\left[\mathbb{I}_{\mathsf{S}_{k-1}^l}(i)|\tfrac{1}{2}g_{k}(X_k^{i,l})+\tfrac{1}{2}g_{k}(X_k^{i,l,a})-
g_{k}(X_k^{i,l-1})|\right].
$$
Then one can apply \cite[Lemma 2.2]{ml_anti} along with Lemmata \ref{lem:cpf1} and \ref{lem:cpf2} to deduce that
$$
\mathbb{E}\left[\mathbb{I}_{\mathsf{S}_{k-1}^l}(i)T_1\right] \leq \frac{C\Delta_l}{N_l}.
$$
For the case of $T_2$ we have the upper-bound
$$
\mathbb{E}\left[\mathbb{I}_{\mathsf{S}_{k-1}^l}(i)T_2\right] 
\leq \frac{C}{N_l^2}\Bigg(
\mathbb{E}\left[\mathbb{I}_{\mathsf{S}_{k-1}^l}(i)|g_{k}(X_k^{i,l,a})-g_{k}(X_k^{i,l})|
\sum_{j\in\mathsf{S}_{k-1}^l}
[g_{k}(X_k^{j,l-1})-g_{k}(X_k^{j,l,a})]
\Big|\right]
+
$$
$$
\mathbb{E}\left[\mathbb{I}_{\mathsf{S}_{k-1}^l}(i)|g_{k}(X_k^{i,l,a})-g_{k}(X_k^{i,l})|
\Big(
N_l - \textrm{Card}(\mathsf{S}_{k-1}^l)
\Big)
\right]
\Bigg).
$$
For the first term on the R.H.S.~one can use the Lipschitz property of $g_{k}$, Cauchy-Schwarz and Lemma \ref{lem:cpf1} and for the second term, the boundedness of $g_k$ to yield that
$$
\mathbb{E}\left[\mathbb{I}_{\mathsf{S}_{k-1}^l}(i)T_2\right] \leq C\left(\frac{\Delta_l}{N_l}+ \frac{1}{N_l}\left\{1-\mathbb{E}\left[\frac{\textrm{Card}(\mathsf{S}_{k-1}^l)}{N_l}\right]\right\}\right)
$$
which concludes the proof.
\end{proof}

\begin{rem}\label{rem:cpf_weight1} 
As for the case of Lemma \ref{lem:cpf_weight} one can establish the following result.  Assume (A\ref{ass:diff1}-\ref{ass:g}). Then for any $k\in\mathbb{N}$ there exists a $C<+\infty$ such that for any $(l,N_l)\in\mathbb{N}^2$
$$
\mathbb{E}\left[\mathbb{I}_{\mathsf{S}_{k-1}^l}(i)\left|\left\{\frac{\tfrac{1}{2}g_{k}(X_k^{i,l})}{\sum_{j=1}^{N_l}g_{k}(X_k^{j,l})}
+
\frac{\tfrac{1}{2}g_{k}(X_k^{i,l,a})}{\sum_{j=1}^{N_l}g_{k}(X_k^{j,l,a})}
-
\frac{g_{k}(X_k^{i,l-1})}{\sum_{j=1}^{N_l}g_{k}(X_k^{j,l-1})}\right\}
\right|^2\right] \leq 
$$
$$
C\left(\frac{\Delta_l}{N_l^2}+ 
\frac{1}{N_l^2}\left\{
1-\mathbb{E}\left[\frac{\textrm{\emph{Card}}(\mathsf{S}_{k-1}^l)}{N_l}\right]\right\}\right).
$$
\end{rem}

\begin{lem}\label{lem:card_lem}
Assume (A\ref{ass:diff1}-\ref{ass:g}). Then for any $k\in\mathbb{N}_0$ there exists a $C<+\infty$ such that for any $(l,N_l,\varepsilon)\in\mathbb{N}^2\times(0,1/2)$
$$
1-\mathbb{E}\left[\frac{\textrm{\emph{Card}}(\mathsf{S}_{k}^l)}{N_l}\right] \leq C\left(\Delta_l + \frac{\Delta_l^{\tfrac{1}{2}-\varepsilon}}{N_l}\right)
$$
\end{lem}

\begin{proof}
We note that
\begin{eqnarray}
1-\mathbb{E}\left[\frac{\textrm{Card}(\mathsf{S}_{k}^l)}{N_l}\right] & = & 
1 - \mathbb{E}\left[\sum_{i=1}^{N_l} \min\left\{\frac{g_k(X_k^{i,l})}{\sum_{j=1}^{N_l}g_k(X_k^{j,l})}
\frac{g_k(X_k^{i,l,a})}{\sum_{j=1}^{N_l}g_k(X_k^{j,l,a})},
\frac{g_k(X_k^{i,l-1})}{\sum_{j=1}^{N_l}g_k(X_k^{j,l-1})}
\right\}\right] + \nonumber\\ & &
 \mathbb{E}\left[\sum_{i\notin\mathsf{S}_{k-1}^l} 
\min\left\{\frac{g_k(X_k^{i,l})}{\sum_{j=1}^{N_l}g_k(X_k^{j,l})}
\frac{g_k(X_k^{i,l,a})}{\sum_{j=1}^{N_l}g_k(X_k^{j,l,a})},
\frac{g_k(X_k^{i,l-1})}{\sum_{j=1}^{N_l}g_k(X_k^{j,l-1})}
\right\}\right]\nonumber\\
& \leq &
1 - \mathbb{E}\left[\sum_{i=1}^{N_l} \min\left\{\frac{g_k(X_k^{i,l})}{\sum_{j=1}^{N_l}g_k(X_k^{j,l})}
\frac{g_k(X_k^{i,l,a})}{\sum_{j=1}^{N_l}g_k(X_k^{j,l,a})},
\frac{g_k(X_k^{i,l-1})}{\sum_{j=1}^{N_l}g_k(X_k^{j,l-1})}
\right\}\right] + \nonumber\\ & &C\left(
1-\mathbb{E}\left[\frac{\textrm{Card}(\mathsf{S}_{k-1}^l)}{N_l}\right]
\right).\label{eq:card_lem1}
\end{eqnarray}
As our proof strategy is one by induction (the initialization is trivally true as $\mathsf{S}_{0}^l=\{1,\dots,N_l\}$), then we need to focus upon the first term on the R.H.S.~of
\eqref{eq:card_lem1}.

Now using the result that for any $(a,b)\in\mathbb{R}^2$, $\min\{a,b\} = \tfrac{1}{2}(a+b-|a-b|)$, twice,  we easily obtain that 
$$
1 - \sum_{i=1}^{N_l} \min\left\{\frac{g_k(X_k^{i,l})}{\sum_{j=1}^{N_l}g_k(X_k^{j,l})}
\frac{g_k(X_k^{i,l,a})}{\sum_{j=1}^{N_l}g_k(X_k^{j,l,a})},
\frac{g_k(X_k^{i,l-1})}{\sum_{j=1}^{N_l}g_k(X_k^{j,l-1})}
\right\}\ = 
$$
$$
\frac{1}{2}\sum_{i=1}^{N_l}
\Bigg\{
\frac{1}{2}\Bigg|
\frac{g_k(X_k^{i,l})}{\sum_{j=1}^{N_l}g_k(X_k^{j,l})} - \frac{g_k(X_k^{i,l,a})}{\sum_{j=1}^{N_l}g_k(X_k^{j,l,a})}
\Bigg| + 
\Bigg|
\frac{\tfrac{1}{2}g_k(X_k^{i,l})}{\sum_{j=1}^{N_l}g_k(X_k^{j,l})} +
\frac{\tfrac{1}{2}g_k(X_k^{i,l,a})}{\sum_{j=1}^{N_l}g_k(X_k^{j,l,a})}
-
\frac{g_k(X_k^{i,l-1})}{\sum_{j=1}^{N_l}g_k(X_k^{j,l-1})}
-
$$
$$
\frac{1}{2}\Bigg|
\frac{g_k(X_k^{i,l})}{\sum_{j=1}^{N_l}g_k(X_k^{j,l})} - \frac{g_k(X_k^{i,l,a})}{\sum_{j=1}^{N_l}g_k(X_k^{j,l,a})}
\Bigg|
\Bigg|
\Bigg\}.
$$
To shorten the subsequent notations, we set
\begin{eqnarray*}
\alpha_i & = & \frac{1}{2}\Bigg|
\frac{g_k(X_k^{i,l})}{\sum_{j=1}^{N_l}g_k(X_k^{j,l})} - \frac{g_k(X_k^{i,l,a})}{\sum_{j=1}^{N_l}g_k(X_k^{j,l,a})}
\Bigg| \\
\beta_i & = & \frac{\tfrac{1}{2}g_k(X_k^{i,l})}{\sum_{j=1}^{N_l}g_k(X_k^{j,l})} +
\frac{\tfrac{1}{2}g_k(X_k^{i,l,a})}{\sum_{j=1}^{N_l}g_k(X_k^{j,l,a})}
-
\frac{g_k(X_k^{i,l-1})}{\sum_{j=1}^{N_l}g_k(X_k^{j,l-1})}.
\end{eqnarray*}
Then in this notation we have that
$$
1 - \mathbb{E}\left[\sum_{i=1}^{N_l} \min\left\{\frac{g_k(X_k^{i,l})}{\sum_{j=1}^{N_l}g_k(X_k^{j,l})}
\frac{g_k(X_k^{i,l,a})}{\sum_{j=1}^{N_l}g_k(X_k^{j,l,a})},
\frac{g_k(X_k^{i,l-1})}{\sum_{j=1}^{N_l}g_k(X_k^{j,l-1})}
\right\}\right]  =  \frac{1}{2}\mathbb{E}\left[\sum_{i=1}^{N_l}\{\alpha_i + |\beta_i-\alpha_i|\}\right]
$$
\begin{equation}
 \leq  \frac{1}{2}\mathbb{E}\left[\sum_{i\in\mathsf{S}_{k-1}^l}\{\alpha_i + |\beta_i-\alpha_i|\}\right]
+ C\left(
1-\mathbb{E}\left[\frac{\textrm{Card}(\mathsf{S}_{k-1}^l)}{N_l}\right]
\right).\label{eq:card_lem2}
\end{equation}
Then by induction, via Lemmata \ref{lem:cpf_weight} and \ref{lem:cpf_weight1} it trivally follows that 
$$
1-\mathbb{E}\left[\frac{\textrm{Card}(\mathsf{S}_{k}^l)}{N_l}\right] \leq C\Delta_l^{1/2}.
$$
Now although this result is not the one we want to prove,  one then obtains from Remarks \ref{rem:cpf_weight} and \ref{rem:cpf_weight1} the following 
\begin{eqnarray}
\mathbb{E}\left[\mathbb{I}_{\mathsf{S}_{k-1}^l}(i)\left|\left\{\frac{g_{k}(X_k^{i,l})}{\sum_{j=1}^{N_l}g_{k}(X_k^{j,l})}-
\frac{g_{k}(X_k^{i,l,a})}{\sum_{j=1}^{N_l}g_{k}(X_k^{j,l,a})}\right\}
\right|^2\right] & \leq & \frac{C\Delta_l^{1/2}}{N_l^2} \label{eq:card_lem4}\\
\mathbb{E}\left[\mathbb{I}_{\mathsf{S}_{k-1}^l}(i)\left|\left\{\frac{\tfrac{1}{2}g_{k}(X_k^{i,l})}{\sum_{j=1}^{N_l}g_{k}(X_k^{j,l})}
+
\frac{\tfrac{1}{2}g_{k}(X_k^{i,l,a})}{\sum_{j=1}^{N_l}g_{k}(X_k^{j,l,a})}
-
\frac{g_{k}(X_k^{i,l-1})}{\sum_{j=1}^{N_l}g_{k}(X_k^{j,l-1})}\right\}
\right|^2\right] & \leq & \frac{C\Delta_l^{1/2}}{N_l^2}.\label{eq:card_lem5}
\end{eqnarray}

Now in order to conclude the proof, we have the upper-bound
$$
1-\mathbb{E}\left[\frac{\textrm{Card}(\mathsf{S}_{k}^l)}{N_l}\right]
 \leq  \frac{1}{2}\mathbb{E}\left[\sum_{i\in\mathsf{S}_{k-1}^l}\{\alpha_i + |\beta_i-\alpha_i|\}\right]
+ C\left(
1-\mathbb{E}\left[\frac{\textrm{Card}(\mathsf{S}_{k-1}^l)}{N_l}\right]
\right)
$$
from equations \eqref{eq:card_lem1} and \eqref{eq:card_lem2}.  If $\beta_i\geq\alpha_i$ or $\alpha_i\leq|\beta_i|$
then in the summand on the R.H.S.~we get $|\beta_i|$ and from Lemma \ref{lem:cpf_weight1} these terms are (the expectation thereof)
\begin{equation}\label{eq:card_lem3}
\mathcal{O}\left(
\frac{\Delta_l}{N_l} + \frac{1}{N_l}\left(1-\mathbb{E}\left[\frac{\textrm{Card}(\mathsf{S}_{k-1}^l)}{N_l}\right]\right)
\right).
\end{equation}
which are particularly useful. The more problematic case is when $\beta_i\leq\alpha_i$ and $\alpha_i\geq|\beta_i|$
which is the one we now focus upon.  Consider the set $\mathsf{A}_{\varepsilon}=\{|\beta_i|>\Delta_l^{\varepsilon}\}$
then our objective is to bound
\begin{equation}\label{eq:extra1}
\mathbb{E}\left[\alpha_i\mathbb{I}_{\mathsf{S}_{k-1}^l}(i)\mathbb{I}_{\{\alpha_i\geq|\beta_i|\}}\right] = 
\mathbb{E}\left[\alpha_i\mathbb{I}_{\mathsf{S}_{k-1}^l}(i)\mathbb{I}_{\{\alpha_i\geq|\beta_i|\}\cap \mathsf{A}_{\varepsilon}^c}
\right] +
\mathbb{E}\left[\alpha_i\mathbb{I}_{\mathsf{S}_{k-1}^l}(i)\mathbb{I}_{\{\alpha_i\geq|\beta_i|\}\cap \mathsf{A}_{\varepsilon}}
\right]
\end{equation}
as all other terms are the order as in \eqref{eq:card_lem3}. For the first term on the R.H.S.~of \eqref{eq:extra1} we must have that
$\alpha_i\geq \Delta_l^{\varepsilon}$ and so it follows that
$$
\mathbb{E}\left[\alpha_i\mathbb{I}_{\mathsf{S}_{k-1}^l}(i)\mathbb{I}_{\{\alpha_i\geq|\beta_i|\}\cap \mathsf{A}_{\varepsilon}^c}
\right] \leq \Delta_l^{-\epsilon}\mathbb{E}\left[\alpha_i^2\mathbb{I}_{\mathsf{S}_{k-1}^l}(i)\right] \leq
\frac{C\Delta_l^{\tfrac{1}{2}-\epsilon}}{N_l^2}
$$
via \eqref{eq:card_lem4}. For the second term on the R.H.S.~of \eqref{eq:extra1}, via Cauchy Schwarz we have
$$
\mathbb{E}\left[\alpha_i\mathbb{I}_{\mathsf{S}_{k-1}^l}(i)\mathbb{I}_{\{\alpha_i\geq|\beta_i|\}\cap \mathsf{A}_{\varepsilon}}
\right]
\leq \mathbb{E}\left[\alpha_i^2\mathbb{I}_{\mathsf{S}_{k-1}^l}(i)\right]^{1/2}
\mathbb{E}\left[\mathbb{I}_{\mathsf{A}_{\varepsilon}}\mathbb{I}_{\mathsf{S}_{k-1}^l}(i)\right]^{1/2}
$$
then using that $\beta_i^2> \Delta_l^{2\varepsilon}$ on $\mathsf{A}_{\varepsilon}$, we have
$$
\mathbb{E}\left[\alpha_i\mathbb{I}_{\mathsf{S}_{k-1}^l}(i)\mathbb{I}_{\{\alpha_i\geq|\beta_i|\}\cap \mathsf{A}_{\varepsilon}}
\right]
\leq \mathbb{E}\left[\alpha_i^2\mathbb{I}_{\mathsf{S}_{k-1}^l}(i)\right]^{1/2}
\mathbb{E}\left[\beta_i^2\mathbb{I}_{\mathsf{S}_{k-1}^l}(i)\right]^{1/2}\Delta_l^{-\varepsilon}
$$
and application of \eqref{eq:card_lem4}-\eqref{eq:card_lem5} gives that
$$
\mathbb{E}\left[\alpha_i\mathbb{I}_{\mathsf{S}_{k-1}^l}(i)\mathbb{I}_{\{\alpha_i\geq|\beta_i|\}\cap \mathsf{A}_{\varepsilon}}
\right] \leq 
\frac{C\Delta_l^{\tfrac{1}{2}-\epsilon}}{N_l^2}.
$$
Hence, we conclude that
$$
1-\mathbb{E}\left[\frac{\textrm{Card}(\mathsf{S}_{k}^l)}{N_l}\right]
 \leq C\left(\Delta_l + \frac{\Delta_l^{\tfrac{1}{2}-\varepsilon}}{N_l}
+1-\mathbb{E}\left[\frac{\textrm{Card}(\mathsf{S}_{k-1}^l)}{N_l}\right]
\right)
$$
and the proof follows by induction.
\end{proof}

\begin{lem}\label{lem:avg_part}
Assume (A\ref{ass:diff1}-\ref{ass:g}). Then for any $k\in\mathbb{N}$ there exists a $C<+\infty$ such that for any $(l,N_l,\varepsilon,i)\in\mathbb{N}^2\times(0,1/2)\times\{1,\dots,N_l\}$
$$
\mathbb{E}\left[\min\{\|\overline{X}_k^{i,l}-X_k^{i,l-1}\|^p,1\}\right] \leq C\left(\Delta_l
+ \frac{\Delta_l^{\tfrac{1}{2}-\varepsilon}}{N_l}
\right).
$$
\end{lem}

\begin{proof}
Follows from Lemmata \ref{lem:cpf2} and \ref{lem:card_lem}; see e.g.~the same proof in \cite[Theorem D.5]{mlpf}.
\end{proof}

\begin{rem}\label{rem:larate}
In a similar manner to the proofs in Lemma \ref{lem:avg_part} and \cite[Theorem D.5]{mlpf}, one can establish the following.
Assume (A\ref{ass:diff1}-\ref{ass:g}):
\begin{itemize}
\item{For any $k\in\mathbb{N}$ there exists a $C<+\infty$ such that for any $(l,N_l,\varepsilon,i)\in\mathbb{N}^2\times(0,1/2)\times\{1,\dots,N_l\}$
$$
\mathbb{E}\left[\min\{\|X_k^{i,l}-X_k^{i,l-1}\|^p,1\}\right] \leq C\left(\Delta_l^{1/2}
+ \frac{\Delta_l^{\tfrac{1}{2}-\varepsilon}}{N_l}
\right).
$$
}
\item{For any $k\in\mathbb{N}$ there exists a $C<+\infty$ such that for any $(l,N_l,\varepsilon,i)\in\mathbb{N}^2\times(0,1/2)\times\{1,\dots,N_l\}$
$$
\mathbb{E}\left[\min\{\|X_k^{i,l}-X_k^{i,l,a}\|^p,1\}\right] \leq C\left(\Delta_l^{1/2}
+ \frac{\Delta_l^{\tfrac{1}{2}-\varepsilon}}{N_l}
\right).
$$
}
\end{itemize}
\end{rem}

\subsection{Particle Convergence Proofs: $\mathbb{L}_p-$Bounds}\label{app:part_lp}

For a function $\varphi\in\mathcal{B}_b(\mathsf{X})\cap\textrm{Lip}(\mathsf{X})$, we set
$\overline{\|\varphi\|}=\max\{\|\varphi\|_{\infty},\|\varphi\|_{\textrm{Lip}}\}$, where $\|\cdot\|_{\textrm{Lip}}$ is the Lipschitz constant.
We denote the $N_l-$empirical measures: $\eta_k^{N_l,l}(dx)= \tfrac{1}{N}\sum_{i=1}^{N_l}\delta_{\{X_k^{i,l}\}}(dx)$, 
$\eta_k^{N_l,l-1}(dx)= \tfrac{1}{N}\sum_{i=1}^{N_l}\delta_{\{X_k^{i,l-1}\}}(dx)$, $\eta_k^{N_l,l,a}(dx)= \tfrac{1}{N}\sum_{i=1}^{N_l}\delta_{\{X_k^{i,l,a}\}}(dx)$ where $k\in\mathbb{N}$ and $\delta_A(dx)$ is the Dirac mass on a set $A$. We define the predictors for $l\in\mathbb{N}_0$, $\eta_1^{l}(dx) = P^l(x_0,dx)$ and for $k\geq 2$, $\eta_k^{l}(dx) =\eta_{k-1}^l(g_{k-1}P^l(dx))/\eta_{k-1}^l(g_{k-1})$.

\begin{lem}\label{lem:ml_conv1}
Assume (A\ref{ass:diff1}-\ref{ass:g}). Then for any $(k,p)\in\mathbb{N}\times[1,\infty)$ there exists a $C<+\infty$ such that for any $(l,N_l,\varepsilon,\varphi)\in\mathbb{N}^2\times(0,1/2)\times\mathcal{B}_b(\mathsf{X})\cap\textrm{\emph{Lip}}(\mathsf{X})$
$$
\mathbb{E}\left[|[\eta_k^{N_l,l}-\eta_k^{N_l,l-1}](\varphi)-[\eta_k^l-\eta_k^{l-1}](\varphi)|^p\right] \leq C\overline{\|\varphi\|}^p\left(\frac{ \Delta_l^{\tfrac{1}{2}}}{N_l^{\tfrac{p}{2}}}
+ \frac{\Delta_l^{\tfrac{1}{2}-\varepsilon}}{N_l^{\tfrac{p}{2}+1}}
\right).
$$
\end{lem}

\begin{proof}
The proof is by induction, with initialization following easily by the Marcinkiewicz Zygmund (M-Z) inequality (the case $p\in[1,2)$ can be dealt with using the bound $p\in[2,\infty)$ and Jensen) and strong convergence results for Euler discretizations. As a result, we proceed to the induction step. We have that
$$
\mathbb{E}\left[|[\eta_k^{N_l,l}-\eta_k^{N_l,l-1}](\varphi)-[\eta_k^l-\eta_k^{l-1}](\varphi)|^p\right]^{1/p} \leq T_1 + T_2
$$
where
\begin{eqnarray*}
T_1 & = & \mathbb{E}\left[|[\eta_k^{N_l,l}-\eta_k^{N_l,l-1}](\varphi) - \mathbb{E}[[\eta_k^{N_l,l}-\eta_k^{N_l,l-1}](\varphi)|\mathcal{F}_{k-1}^l]|^p\right]^{1/p} \\
T_2 & = & \mathbb{E}\left[|
\mathbb{E}[[\eta_k^{N_l,l}-\eta_k^{N_l,l-1}](\varphi)|\mathcal{F}_{k-1}]-[\eta_k^l-\eta_k^{l-1}](\varphi)|^p\right]^{1/p}.
\end{eqnarray*}
We deal with the two terms, in turn. For the case of $T_1$, one can apply the (conditional) M-Z inequality to yield
$$
T_1 \leq \frac{C}{\sqrt{N_l}} \mathbb{E}\left[|\varphi(X_k^{1,l})-\varphi(X_k^{1,l-1})|^p\right]^{1/p}.
$$
using $\varphi\in\mathcal{B}_b(\mathsf{X})\cap\textrm{Lip}(\mathsf{X})$ along with Remark \ref{rem:larate}
yields
\begin{equation}\label{eq:part_conv1}
T_1 \leq \frac{C\overline{\|\varphi\|}}{\sqrt{N_l}}\left(\Delta_l^{\tfrac{1}{2}}
+ \frac{\Delta_l^{\tfrac{1}{2}-\varepsilon}}{N_l}
\right)^{1/p}.
\end{equation}

For $T_2$, we first note that
$$
\mathbb{E}[[\eta_k^{N_l,l}-\eta_k^{N_l,l-1}](\varphi)|\mathcal{F}_{k-1}^l]-[\eta_k^l-\eta_k^{l-1}](\varphi) = 
$$
$$
\left\{\frac{\eta_{k-1}^{N_l,l}(g_{k-1}P^l(\varphi))}{\eta_{k-1}^{N_l,l}(g_{k-1})} - 
\frac{\eta_{k-1}^{N_l,l-1}(g_{k-1}P^{l-1}(\varphi))}{\eta_{k-1}^{N_l,l-1}(g_{k-1})}\right\} -
\left\{\frac{\eta_{k-1}^{l}(g_{k-1}P^l(\varphi))}{\eta_{k-1}^{l}(g_{k-1})} - 
\frac{\eta_{k-1}^{l-1}(g_{k-1}P^{l-1}(\varphi))}{\eta_{k-1}^{l-1}(g_{k-1})}\right\}
$$
where we recall that $P^l$ (resp.~$P^{l-1}$) represents the truncated Milstein kernel over unit time with discretization level $\Delta_l$
(resp.~$\Delta_{l-1}$). This can be further decomposed to 
$$
\mathbb{E}[[\eta_k^{N_l,l}-\eta_k^{N_l,l-1}](\varphi)|\mathcal{F}_{k-1}]-[\eta_k^l-\eta_k^{l-1}](\varphi) = 
$$
$$
\Bigg\{\frac{\eta_{k-1}^{N_l,l}(g_{k-1}[P^l-P](\varphi))}{\eta_{k-1}^{N_l,l}(g_{k-1})} - 
\frac{\eta_{k-1}^{l}(g_{k-1}[P^l-P](\varphi))}{\eta_{k-1}^{l}(g_{k-1})}\Bigg\} - 
\left\{\frac{\eta_{k-1}^{l}(g_{k-1}[P^{l-1}-P](\varphi))}{\eta_{k-1}^{l}(g_{k-1})} - 
\frac{\eta_{k-1}^{l-1}(g_{k-1}[P^{l-1}-P](\varphi))}{\eta_{k-1}^{l-1}(g_{k-1})}\right\} + 
$$
$$
\left\{\frac{\eta_{k-1}^{N_l,l}(g_{k-1}P(\varphi))}{\eta_{k-1}^{N_l,l}(g_{k-1})} - 
\frac{\eta_{k-1}^{N_l,l-1}(g_{k-1}P(\varphi))}{\eta_{k-1}^{N_l,l-1}(g_{k-1})}\right\} -
\left\{\frac{\eta_{k-1}^{l}(g_{k-1}P(\varphi))}{\eta_{k-1}^{l}(g_{k-1})} - 
\frac{\eta_{k-1}^{l-1}(g_{k-1}P(\varphi))}{\eta_{k-1}^{l-1}(g_{k-1})}\right\} 
$$
Therefore we have that
$$
T_2 \leq T_3 + T_4 + T_5
$$
where
\begin{eqnarray}
T_3 & = & \mathbb{E}\left[\left|\frac{\eta_{k-1}^{N_l,l}(g_{k-1}[P^l-P](\varphi))}{\eta_{k-1}^{N_l,l}(g_{k-1})} - 
\frac{\eta_{k-1}^{l}(g_{k-1}[P^l-P](\varphi))}{\eta_{k-1}^{l}(g_{k-1})}\right|^{p}\right]^{1/p} \label{eq:part_conv4}\\
T_4 & = & \mathbb{E}\left[\left|\frac{\eta_{k-1}^{N_l,l-1}(g_{k-1}[P^{l-1}-P](\varphi))}{\eta_{k-1}^{N_l,l-1}(g_{k-1})} - 
\frac{\eta_{k-1}^{l-1}(g_{k-1}[P^{l-1}-P](\varphi))}{\eta_{k-1}^{l-1}(g_{k-1})}\right|^{p}\right]^{1/p} \nonumber\\
T_5 & = & \mathbb{E}\Bigg[\Bigg|
\left\{\frac{\eta_{k-1}^{N_l,l}(g_{k-1}P(\varphi))}{\eta_{k-1}^{N_l,l}(g_{k-1})} - 
\frac{\eta_{k-1}^{N_l,l-1}(g_{k-1}P(\varphi))}{\eta_{k-1}^{N_l,l-1}(g_{k-1})}\right\} -
\nonumber\\ & &
\left\{\frac{\eta_{k-1}^{l}(g_{k-1}P(\varphi))}{\eta_{k-1}^{l}(g_{k-1})} - 
\frac{\eta_{k-1}^{l-1}(g_{k-1}P(\varphi))}{\eta_{k-1}^{l-1}(g_{k-1})}\right\}\Bigg|^{p}\Bigg]^{1/p}.\nonumber
\end{eqnarray}
$T_3$ and $T_4$ can be treated using similar calculations, so we only consider $T_4$.
For the case of $T_3$, clearly we have the upper-bound
$$
T_3 \leq T_6 + T_7
$$
where
\begin{eqnarray*}
T_6 & = & \mathbb{E}\left[\left|\frac{\eta_{k-1}^{N_l,l}(g_{k-1}[P^l-P](\varphi))-\eta_{k-1}^{l}(g_{k-1}[P^l-P](\varphi))}{\eta_{k-1}^{N_l,l}(g_{k-1})}\right|^{p}\right]^{1/p} \\
T_7 & = & \mathbb{E}\left[\left|\frac{
\eta_{k-1}^{l}(g_{k-1}[P^l-P](\varphi))[\eta_{k-1}^{l}-\eta_{k-1}^{N_l,l}](g_{k-1})
}{\eta_{k-1}^{N_l,l}(g_{k-1})\eta_{k-1}^{l}(g_{k-1})}\right|^{p}\right]^{1/p}.
\end{eqnarray*}
For $T_6$ we have that 
$$
T_6 \leq C\mathbb{E}\left[\left|[\eta_{k-1}^{N_l,l}-\eta_{k-1}^{l}](g_{k-1}[P^l-P](\varphi))\right|^{p}\right]^{1/p}.
$$
By using standard results for particle filters (e.g.~\cite[Proposition C.6]{mlpf}) we have
$$
T_6 \leq \frac{C\|[P^l-P](\varphi)\|_{\infty}}{\sqrt{N_l}}.
$$
Then using standard results for weak errors (the truncated Milstein scheme is a first-order method) we have
$$
T_7 \leq \frac{C\overline{\|\varphi\|}\Delta_l}{\sqrt{N_l}}.
$$
For $T_7$, again using weak errors 
$$
T_7 \leq C\overline{\|\varphi\|}\Delta_l \mathbb{E}\left[\left|[\eta_{k-1}^{l}-\eta_{k-1}^{N_l,l}](g_{k-1})\right|^p\right]^{1/p}
$$
and again using standard results for particle filters we obtain
$$
T_7 \leq \frac{C\overline{\|\varphi\|}\Delta_l}{\sqrt{N_l}}.
$$
Hence, we have shown that
\begin{equation}\label{eq:part_conv2}
\max\{T_3,T_4\} \leq \frac{C\overline{\|\varphi\|}\Delta_l}{\sqrt{N_l}}.
\end{equation}
For $T_5$, one can apply \cite[Lemma C.5]{mlpf} to show that
$$
T_5 \leq \sum_{j=1}^6 T_{j+7}
$$
where
\begin{eqnarray*}
T_8 & = &  \mathbb{E}\left[\left|\frac{1}{\eta_{k-1}^{N_l,l}(g_{k-1})}
\left\{
[\eta_{k-1}^{N_l,l}-\eta_{k-1}^{N_l,l-1}](g_{k-1}P(\varphi)) -
[\eta_{k-1}^{l}-\eta_{k-1}^{l-1}](g_{k-1}P(\varphi))
\right\}
\right|^p\right]^{1/p} \\
T_9 & = & \mathbb{E}\left[\left|\frac{\eta_{k-1}^{N_l,l-1}(g_{k-1}P(\varphi))}{\eta_{k-1}^{N_l,l}(g_{k-1})
\eta_{k-1}^{N_l,l-1}(g_{k-1})}
\left\{
[\eta_{k-1}^{N_l,l}-\eta_{k-1}^{N_l,l-1}](g_{k-1}) -
[\eta_{k-1}^{l}-\eta_{k-1}^{l-1}](g_{k-1})
\right\}
\right|^p\right]^{1/p}\\
T_{10} & = & \mathbb{E}\left[\left|\frac{1}{\eta_{k-1}^{N_l,l}(g_{k-1})\eta_{k-1}^{l}(g_{k-1})}
\{[\eta_{k-1}^{l}-\eta_{k-1}^{N_l,l}](g_{k-1})\}\{[\eta_{k-1}^{l}-\eta_{k-1}^{l-1}](g_{k-1}P(\varphi))\}
\right|^p\right]^{1/p}\\
T_{11} & = & \mathbb{E}\left[\left|\frac{1}{\eta_{k-1}^{N_l,l}(g_{k-1})\eta_{k-1}^{N_l,l-1}(g_{k-1})}
\{[\eta_{k-1}^{N_l,l-1}-\eta_{k-1}^{l-1}](g_{k-1}P(\varphi))\}\{
[\eta_{k-1}^{l}-\eta_{k-1}^{l-1}](g_{k-1})
\}
\right|^p\right]^{1/p}\\
T_{12} & = & \mathbb{E}\left[\left|\frac{\eta_{k-1}^{l-1}(g_{k-1}P(\varphi))}{
\eta_{k-1}^{l}(g_{k-1})\eta_{k-1}^{N_l,l-1}(g_{k-1})\eta_{k-1}^{l-1}(g_{k-1})}
\{[\eta_{k-1}^{N_l,l-1}-\eta_{k-1}^{l-1}](g_{k-1})\}
\{[\eta_{k-1}^{l}-\eta_{k-1}^{l-1}](g_{k-1})\}
\right|^p\right]^{1/p}\\
T_{13} & = & \mathbb{E}\left[\left|\frac{\eta_{k-1}^{l-1}(g_{k-1}P(\varphi))}{\eta_{k-1}^{N_l,l}(g_{k-1})
\eta_{k-1}^{l}(g_{k-1})\eta_{k-1}^{N_l,l-1}(g_{k-1})}
\{[\eta_{k-1}^{N_l,l}-\eta_{k-1}^{l}](g_{k-1})\}
\{[\eta_{k-1}^{l}-\eta_{k-1}^{l-1}](g_{k-1})\}
\right|^p\right]^{1/p}.
\end{eqnarray*}
Since  $T_8$ and $T_9$ can be bounded using similar approaches, we consider only $T_8$. Similarly
$T_{10},\dots,T_{13}$ can be bounded in almost the same way, so we consider only $T_{10}$.
For $T_8$ we have the upper-bound
$$
T_8 \leq C\mathbb{E}\left[\left|
\left\{
[\eta_{k-1}^{N_l,l}-\eta_{k-1}^{N_l,l-1}](g_{k-1}P(\varphi)) -
[\eta_{k-1}^{l}-\eta_{k-1}^{l-1}](g_{k-1}P(\varphi))
\right\}
\right|^p\right]^{1/p}. 
$$
Since $g_{k-1}P(\varphi)\in\mathcal{B}_b(\mathsf{X})\cap\textrm{Lip}(\mathsf{X})$ (see e.g.~\cite[Eq.(2.6)]{delm})
it follows by the induction hypothesis
$$
T_8 \leq \frac{C\overline{\|\varphi\|}}{\sqrt{N_l}}\left(\Delta_l^{\tfrac{1}{2}}
+ \frac{\Delta_l^{\tfrac{1}{2}-\varepsilon}}{N_l}
\right)^{1/p}.
$$
For $T_{10}$, we have the upper-bound
$$
T_{10} \leq C|[\eta_{k-1}^{l}-\eta_{k-1}^{l-1}](g_{k-1}P(\varphi))|
\mathbb{E}[|([\eta_{k-1}^{l}-\eta_{k-1}^{N_l,l}](g_{k-1}))|^p]^{1/p}.
$$
Then using \cite[Lemma D.2]{mlpf} and standard results for particle filters
$$
T_{10} \leq \frac{C\overline{\|\varphi\|}\Delta_l}{\sqrt{N_l}}.
$$
Therefore we deduce that 
\begin{equation}\label{eq:part_conv3}
T_5 \leq \frac{C\overline{\|\varphi\|}}{\sqrt{N_l}}\left(\Delta_l^{\tfrac{1}{2}}
+ \frac{\Delta_l^{\tfrac{1}{2}-\varepsilon}}{N_l}
\right)^{1/p}.
\end{equation}
The proof is then easily completed by combining \eqref{eq:part_conv1}, \eqref{eq:part_conv2} and \eqref{eq:part_conv3}.
\end{proof}

\begin{rem}\label{rem:ml_conv1}
One can also deduce the following result using a similar (but simpler) proof to Lemma \ref{lem:ml_conv1}.
Assume (A\ref{ass:diff1}-\ref{ass:g}). Then for any $(k,p)\in\mathbb{N}\times[1,\infty)$ there exists a $C<+\infty$ such that for any $(l,N_l,\varepsilon,\varphi)\in\mathbb{N}^2\times(0,1/2)\times\mathcal{B}_b(\mathsf{X})\cap\textrm{\emph{Lip}}(\mathsf{X})$
$$
\mathbb{E}\left[|[\eta_k^{N_l,l}-\eta_k^{N_l,l,a}](\varphi)|^p\right] \leq C\overline{\|\varphi\|}^p\left(\frac{ \Delta_l^{\tfrac{1}{2}}}{N_l^{\tfrac{p}{2}}}
+ \frac{\Delta_l^{\tfrac{1}{2}-\varepsilon}}{N_l^{\tfrac{p}{2}+1}}
\right).
$$
\end{rem}

Below $\mathcal{C}_b^2(\mathsf{X},\mathbb{R})$ denotes the collection of twice continuously differentiable functions from $\mathsf{X}$ to $\mathbb{R}$ with bounded derivatives of all order 1 and 2.

\begin{lem}\label{lem:ml_conv2}
Assume (A\ref{ass:diff1}-\ref{ass:g}). Then for any $(k,p,\varphi)\in\mathbb{N}\times[1,\infty)\times
\mathcal{B}_b(\mathsf{X})\cap\mathcal{C}_b^2(\mathsf{X},\mathbb{R})$ there exists a $C<+\infty$ such that for any $(l,N_l,\varepsilon)\in\mathbb{N}^2\times(0,1/2)$
$$
\mathbb{E}\left[|[\tfrac{1}{2}\eta_k^{N_l,l}+\tfrac{1}{2}\eta_k^{N,l,a}-\eta_k^{N_l,l-1}](\varphi)-[\eta_k^l-\eta_k^{l-1}](\varphi)|^p\right] \leq C\left(\frac{\Delta_l}{N_l^{p/2}} +
\frac{\Delta_l^{\tfrac{1}{2}}}{N_l^p}
+ \frac{\Delta_l^{\tfrac{1}{2}-\varepsilon}}{N_l^{p/2+1}}
\right).
$$
\end{lem}

\begin{proof}
The proof is by induction, with initialization following easily by the M-Z inequality and \cite[Lemma 2.2, Theorem 4.10]{ml_anti}
(the case $p\in[1,2)$ can also be covered with their theory), so we proceed immediately to the induction step.

As for the proof of Lemma \ref{lem:ml_conv1} one can add and subtract the conditional expectation to obtain an upper-bound
$$
\mathbb{E}\left[|[\tfrac{1}{2}\eta_k^{N_l,l}+\tfrac{1}{2}\eta_k^{N,l,a}-\eta_k^{N_l,l-1}](\varphi)-[\eta_k^l-\eta_k^{l-1}](\varphi)|^p\right]^{1/p} \leq T_1 + T_2 
$$
where
\begin{eqnarray*}
T_1 & = & \mathbb{E}\left[|[\tfrac{1}{2}\eta_k^{N_l,l}+\tfrac{1}{2}\eta_k^{N,l,a}-\eta_k^{N_l,l-1}](\varphi)-
\mathbb{E}[\tfrac{1}{2}\eta_k^{N_l,l}+\tfrac{1}{2}\eta_k^{N,l,a}-\eta_k^{N_l,l-1}](\varphi)|\mathcal{F}_{k-1}]|^p\right]^{1/p}\\
T_2 & = &  \mathbb{E}\left[|\mathbb{E}[\tfrac{1}{2}\eta_k^{N_l,l}+\tfrac{1}{2}\eta_k^{N,l,a}-\eta_k^{N_l,l-1}](\varphi)|\mathcal{F}_{k-1}]-[\eta_k^l-\eta_k^{l-1}](\varphi)|^p\right]^{1/p}.
\end{eqnarray*}

For $T_1$ one can use the conditional M-Z inequality along with the boundedness of $\varphi$, \cite[Lemma 2.2]{ml_anti} and Lemma \ref{lem:avg_part} and Remark \ref{rem:larate} to deduce that
\begin{equation}\label{eq:part_conv5} 
T_1 \leq \frac{C}{\sqrt{N_l}}\left(\Delta_l
+ \frac{\Delta_l^{\tfrac{1}{2}-\varepsilon}}{N_l}
\right)^{1/p}.
\end{equation}

The case of $T_2$ is more challenging. We have the following decomposition
$$
T_2 \leq T_3 + T_4
$$
where
\begin{eqnarray}
T_3 & = & \mathbb{E}\Bigg[\Bigg|
\Bigg\{\frac{\tfrac{1}{2}\eta_{k-1}^{N_l,l}(g_{k-1}[P^l-P](\varphi))}{\eta_{k-1}^{N_l,l}(g_{k-1})} + 
\frac{\tfrac{1}{2}\eta_{k-1}^{N_l,l,a}(g_{k-1}[P^l-P](\varphi))}{\eta_{k-1}^{N_l,l,a}(g_{k-1})}  - 
\frac{\eta_{k-1}^{N_l,l-1}(g_{k-1}[P^{l-1}-P](\varphi))}{\eta_{k-1}^{N_l,l-1}(g_{k-1})} 
\Bigg\}
- \nonumber\\ & &
\Bigg\{
\frac{\tfrac{1}{2}\eta_{k-1}^{l}(g_{k-1}[P^l-P](\varphi))}{\eta_{k-1}^{l}(g_{k-1})} + 
\frac{\tfrac{1}{2}\eta_{k-1}^{l}(g_{k-1}[P^l-P](\varphi))}{\eta_{k-1}^{l}(g_{k-1})}  - 
\frac{\eta_{k-1}^{l-1}(g_{k-1}[P^{l-1}-P](\varphi))}{\eta_{k-1}^{l-1}(g_{k-1})} 
\Bigg\}
\Bigg|^p\Bigg]^{1/p} \nonumber\\
T_4 & = & \mathbb{E}\Bigg[\Bigg|
\Bigg\{\frac{\tfrac{1}{2}\eta_{k-1}^{N_l,l}(g_{k-1}P(\varphi))}{\eta_{k-1}^{N_l,l}(g_{k-1})} + 
\frac{\tfrac{1}{2}\eta_{k-1}^{N_l,l,a}(g_{k-1}P(\varphi))}{\eta_{k-1}^{N_l,l,a}(g_{k-1})}  - 
\frac{\eta_{k-1}^{N_l,l-1}(g_{k-1}P(\varphi))}{\eta_{k-1}^{N_l,l-1}(g_{k-1})} 
\Bigg\}
- \nonumber\\ & &
\Bigg\{
\frac{\tfrac{1}{2}\eta_{k-1}^{l}(g_{k-1}P(\varphi))}{\eta_{k-1}^{l}(g_{k-1})} + 
\frac{\tfrac{1}{2}\eta_{k-1}^{l}(g_{k-1}P(\varphi))}{\eta_{k-1}^{l}(g_{k-1})}  - 
\frac{\eta_{k-1}^{l-1}(g_{k-1}P(\varphi))}{\eta_{k-1}^{l-1}(g_{k-1})} 
\Bigg\}
\Bigg|^p\Bigg]^{1/p}.\label{eq:part_conv8} 
\end{eqnarray}
For the case of $T_3$ one has the upper-bound:
$$
T_3 \leq \sum_{j=1}^3 T_{j+4}
$$
where
\begin{eqnarray*}
T_5 & = & \mathbb{E}\Bigg[\Bigg|
\Bigg\{\frac{\tfrac{1}{2}\eta_{k-1}^{N_l,l}(g_{k-1}[P^l-P](\varphi))}{\eta_{k-1}^{N_l,l}(g_{k-1})}  - 
\frac{\tfrac{1}{2}\eta_{k-1}^{l}(g_{k-1}[P^l-P](\varphi))}{\eta_{k-1}^{l}(g_{k-1})}\Bigg|^p\Bigg]^{1/p}\\
T_6 & = & \mathbb{E}\Bigg[\Bigg|
\Bigg\{\frac{\tfrac{1}{2}\eta_{k-1}^{N_l,l,a}(g_{k-1}[P^l-P](\varphi))}{\eta_{k-1}^{N_l,l,a}(g_{k-1})}  - 
\frac{\tfrac{1}{2}\eta_{k-1}^{l}(g_{k-1}[P^l-P](\varphi))}{\eta_{k-1}^{l}(g_{k-1})}\Bigg|^p\Bigg]^{1/p}\\
T_7 & = & 
\mathbb{E}\Bigg[\Bigg|
\Bigg\{\frac{\tfrac{1}{2}\eta_{k-1}^{N_l,l-1}(g_{k-1}[P^{l-1}-P](\varphi))}{\eta_{k-1}^{N_l,l-1}(g_{k-1})}  - 
\frac{\tfrac{1}{2}\eta_{k-1}^{l-1}(g_{k-1}[P^{l-1}-P](\varphi))}{\eta_{k-1}^{l-1}(g_{k-1})}\Bigg|^p\Bigg]^{1/p}.
\end{eqnarray*}
Each of these terms can be controlled (almost) exactly as is done for \eqref{eq:part_conv4} in the proof of Lemma \ref{lem:ml_conv1} and so we do not give the proof, but simply state that
\begin{equation}\label{eq:part_conv6} 
T_3 \leq \frac{C\Delta_l}{\sqrt{N_l}}.
\end{equation}
For $T_4$ one can apply Lemma \ref{lem:tech_lem_rat} along with Minkowski to deduce that
$$
T_4 \leq \sum_{j=1}^4 T_{j+7}
$$
where
\begin{eqnarray*}
T_8 & = & \mathbb{E}\Bigg[\Bigg|\frac{1}{\eta_{k-1}^{N_l,l-1}(g_{k-1})}
[\tfrac{1}{2}\eta_{k-1}^{N_l,l}+\tfrac{1}{2}\eta_{k-1}^{N_l,l,a}-\eta_{k-1}^{N_l,l-1}](g_{k-1}P(\varphi)) -
\frac{1}{\eta_{k-1}^{l-1}(g_{k-1})}
[\tfrac{1}{2}\eta_{k-1}^{l}+\\ & & \tfrac{1}{2}\eta_{k-1}^{l}-\eta_{k-1}^{l-1}](g_{k-1}P(\varphi))
\Bigg|^p\Bigg]^{1/p}\\
T_9 & = &  \mathbb{E}\Bigg[\Bigg|\frac{1}{\eta_{k-1}^{N_l,l}(g_{k-1})\eta_{k-1}^{N_l,l-1}(g_{k-1})}\tfrac{1}{2}
\{[\eta_{k-1}^{N_l,l}-\eta_{k-1}^{N_l,l,a}](g_{k-1}P(\varphi))\}\{[\eta_{k-1}^{N_l,l-1}-\eta_{k-1}^{N_l,l}](g_{k-1})\}
\Bigg|^p\Bigg]^{1/p} \\
T_{10} & = & \mathbb{E}\Bigg[\Bigg|
\frac{\tfrac{1}{2}\eta_{k-1}^{N_l,l,a}(g_{k-1}P(\varphi))}{\eta_{k-1}^{N_l,l,a}(g_{k-1})\eta_{k-1}^{N_l,l}(g_{k-1})\eta_{k-1}^{N_l,l-1}](g_{k-1})}
\{[\eta_{k-1}^{N_l,l,a}-\eta_{k-1}^{N_l,l}](g_{k-1})\}\{[\eta_{k-1}^{N_l,l-1}-\eta_{k-1}^{N_l,l}](g_{k-1})\}
\Bigg|^p\Bigg]^{1/p} \\
T_{11} & = & \mathbb{E}\Bigg[\Bigg|
\frac{\eta_{k-1}^{N_l,l,a}(g_{k-1}P(\varphi))}{\eta_{k-1}^{N_l,l,a}(g_{k-1})\eta_{k-1}^{N_l,l-1}(g_{k-1})}
[\tfrac{1}{2}\eta_{k-1}^{N_l,l}+\tfrac{1}{2}\eta_{k-1}^{N_l,l,a}-\eta_{k-1}^{N_l,l-1}](g_{k-1}) - 
\frac{\eta_{k-1}^{l}(g_{k-1}P(\varphi))}{\eta_{k-1}^{l}(g_{k-1})\eta_{k-1}^{l-1}(g_{k-1})}\times \\ & &
[\tfrac{1}{2}\eta_{k-1}^{l}+\tfrac{1}{2}\eta_{k-1}^{l}-\eta_{k-1}^{l-1}](g_{k-1})
\Bigg|^p\Bigg]^{1/p}.
\end{eqnarray*}
As $T_8$ (resp.~$T_9$) can be dealt with using similar arguments to $T_{11}$ (resp.~$T_{10}$) we only prove bounds for $T_8$ (resp.~$T_9$). For $T_8$ we have the upper-bound
$$
T_8 \leq T_{12} + T_{13}
$$
where
\begin{eqnarray*}
T_{12} & = & \mathbb{E}\Bigg[\Bigg|
\frac{1}{\eta_{k-1}^{N_l,l-1}(g_{k-1})}\Bigg\{
[\tfrac{1}{2}\eta_{k-1}^{N_l,l}+\tfrac{1}{2}\eta_{k-1}^{N_l,l,a}-\eta_{k-1}^{N_l,l-1}](g_{k-1}P(\varphi)) - 
[\tfrac{1}{2}\eta_{k-1}^{l}+\tfrac{1}{2}\eta_{k-1}^{l}-\\ & &\eta_{k-1}^{l-1}](g_{k-1}P(\varphi))\Bigg\}
\Bigg|^p\Bigg]^{1/p} \\
T_{13} & = & \mathbb{E}\Bigg[\Bigg|\frac{[\tfrac{1}{2}\eta_{k-1}^{l}+\tfrac{1}{2}\eta_{k-1}^{l}-\eta_{k-1}^{l-1}](g_{k-1}P(\varphi))}{\eta_{k-1}^{N_l,l-1}(g_{k-1})\eta_{k-1}^{l-1}(g_{k-1})}[\eta_{k-1}^{l-1}-\eta_{k-1}^{N_l,l-1}](g_{k-1})
\Bigg|^p\Bigg]^{1/p}.
\end{eqnarray*}
For $T_{12}$ one can use the lower-bound on $g_{k-1}$ along with $g_{k-1}P(\varphi)\in\mathcal{C}_b^2(\mathsf{X},\mathbb{R})$ (see e.g.~\cite[Corollary 2.2.8]{stroock}) and the induction hypothesis to yield
$$
T_{12} \leq \frac{C}{\sqrt{N_l}}\left(\Delta_l
+ \frac{\Delta_l^{\tfrac{1}{2}-\varepsilon}}{N_l}
\right)^{1/p}.
$$
For $T_{13}$ we apply \cite[Lemma D.2]{mlpf} and standard results for particle filters to obtain
$$
T_{13} \leq \frac{C\Delta_l}{\sqrt{N_l}}.
$$
So that 
$$
T_{8} \leq \frac{C}{\sqrt{N_l}}\left(\Delta_l
+ \frac{\Delta_l^{\tfrac{1}{2}-\varepsilon}}{N_l}
\right)^{1/p}.
$$
For $T_9$ we have the upper-bound
$$
T_9 \leq T_{14} + T_{15}
$$
where
\begin{eqnarray*}
T_{14} & = & \mathbb{E}\Bigg[\Bigg|
\frac{1}{\eta_{k-1}^{N_l,l}(g_{k-1})\eta_{k-1}^{N_l,l-1}(g_{k-1})}\tfrac{1}{2}
\{[\eta_{k-1}^{N_l,l}-\eta_{k-1}^{N_l,l,a}](g_{k-1}P(\varphi))\}\{[\eta_{k-1}^{N_l,l-1}-\eta_{k-1}^{N_l,l}](g_{k-1})
-\\& &[\eta_{k-1}^{l-1}-\eta_{k-1}^{l}](g_{k-1})\}
\Bigg|^p\Bigg]^{1/p} \\
T_{15} & = & \mathbb{E}\Bigg[\Bigg|
\frac{1}{\eta_{k-1}^{N_l,l}(g_{k-1})\eta_{k-1}^{N_l,l-1}(g_{k-1})}\tfrac{1}{2}
\{[\eta_{k-1}^{N_l,l}-\eta_{k-1}^{N_l,l,a}](g_{k-1}P(\varphi))\}\{
[\eta_{k-1}^{l-1}-\eta_{k-1}^{l}](g_{k-1})\}
\Bigg|^p\Bigg]^{1/p}.
\end{eqnarray*}
For $T_{14}$ one can use the lower-bound on $g_{k-1}$, Cauchy-Schwarz, Remark \ref{rem:ml_conv1} and Lemma \ref{lem:ml_conv1} to give
$$
T_{14} \leq \frac{C}{N_l}\left(\Delta_l
+ \frac{\Delta_l^{\tfrac{1}{2}-\varepsilon}}{N_l}
\right)^{1/p}.
$$
For $T_{15}$ one can use the lower-bound on $g_{k-1}$, \cite[Lemma D.2]{mlpf} and Remark \ref{rem:ml_conv1} to give
$$
T_{15} \leq \frac{C}{\sqrt{N_l}}\left(\Delta_l^{p+1/2}
+ \frac{\Delta_l^{p+\tfrac{1}{2}-\varepsilon}}{N_l}
\right)^{1/p}
$$
and thus
$$
T_{9} \leq \frac{C}{\sqrt{N_l}}\left(\Delta_l
+ \frac{\Delta_l^{\tfrac{1}{2}-\varepsilon}}{N_l}
\right)^{1/p} + \frac{C}{N_l}\left(\Delta_l^{1/2}
+ \frac{\Delta_l^{\tfrac{1}{2}-\varepsilon}}{N_l}
\right)^{1/p}  .
$$
Therefore we have shown that
\begin{equation}\label{eq:part_conv7} 
T_{4} \leq \frac{C}{\sqrt{N_l}}\left(\Delta_l + \frac{\Delta_l^{\tfrac{1}{2}}}{N_l^{p/2}}
+ \frac{\Delta_l^{\tfrac{1}{2}-\varepsilon}}{N_l}
\right)^{1/p}.
\end{equation}
The proof can be completed by combining the bounds \eqref{eq:part_conv5}, \eqref{eq:part_conv6} and \eqref{eq:part_conv7}.
\end{proof}

\begin{theorem}\label{theo:var}
Assume (A\ref{ass:diff1}-\ref{ass:g}). Then for any $(k,p,\varphi)\in\mathbb{N}\times[1,\infty)\times
\mathcal{B}_b(\mathsf{X})\cap\mathcal{C}_b^2(\mathsf{X},\mathbb{R})$ there exists a $C<+\infty$ such that for any $(l,N_l,\varepsilon)\in\mathbb{N}^2\times(0,1/2)$
$$
\mathbb{E}\left[|[\pi_k^l-\pi_{k}^{l-1}]^{N_l}(\varphi) - 
[\pi_k^l-\pi_{k}^{l-1}](\varphi)
|^p\right] \leq  C\left(\frac{\Delta_l}{N_l^{p/2}} +
\frac{\Delta_l^{\tfrac{1}{2}}}{N_l^p}
+ \frac{\Delta_l^{\tfrac{1}{2}-\varepsilon}}{N_l^{p/2+1}}
\right).
$$
\end{theorem}

\begin{proof}
This follows from Lemmata \ref{lem:ml_conv1}, \ref{lem:ml_conv2}, Remark \ref{rem:ml_conv1} and a similar approach to control \eqref{eq:part_conv8} in the proof of Lemma \ref{lem:ml_conv2}; we omit the details.
\end{proof}

\subsection{Particle Convergence Proofs: Bias Bounds}\label{app:part_bias}

\begin{lem}\label{lem:bias1}
Assume (A\ref{ass:diff1}-\ref{ass:g}). Then for any $k\in\mathbb{N}$ there exists a $C<+\infty$ such that for any $(l,N_l,\varepsilon,\varphi)\in\mathbb{N}^2\times(0,1/2)\times\mathcal{B}_b(\mathsf{X})\cap\textrm{\emph{Lip}}(\mathsf{X})$
$$
\left|\mathbb{E}\left[[\eta_k^{N_l,l}-\eta_k^{N_l,l-1}](\varphi)-[\eta_k^l-\eta_k^{l-1}](\varphi)\right]\right| \leq \frac{C\overline{\|\varphi\|}}{N_l}
\left(\Delta_l^{\tfrac{1}{2}}
+ \frac{\Delta_l^{\tfrac{1}{2}-\varepsilon}}{N_l}
\right)^{\tfrac{1}{2}}.
$$
\end{lem}

\begin{proof}
The proof is by induction, with the case $k=1$ trivial as there is no bias, so we proceed to the induction step. Following the proof of Lemma \ref{lem:ml_conv1} we have the decomposition
$$
\left|\mathbb{E}\left[[\eta_k^{N_l,l}-\eta_k^{N_l,l-1}](\varphi)-[\eta_k^l-\eta_k^{l-1}](\varphi)\right]\right| \leq T_1+ T_2 + T_3
$$
where
\begin{eqnarray*}
T_1 & = & \left|\mathbb{E}\left[\frac{\eta_{k-1}^{N_l,l}(g_{k-1}[P^l-P](\varphi))}{\eta_{k-1}^{N_l,l}(g_{k-1})} - 
\frac{\eta_{k-1}^{l}(g_{k-1}[P^l-P](\varphi))}{\eta_{k-1}^{l}(g_{k-1})}\right]\right| \\
T_2 & = & \left|\mathbb{E}\left[\frac{\eta_{k-1}^{N_l,l-1}(g_{k-1}[P^{l-1}-P](\varphi))}{\eta_{k-1}^{N_l,l-1}(g_{k-1})} - 
\frac{\eta_{k-1}^{l-1}(g_{k-1}[P^{l-1}-P](\varphi))}{\eta_{k-1}^{l-1}(g_{k-1})}\right]\right| \\
T_3 & = & \Bigg|\mathbb{E}\Bigg[
\left\{\frac{\eta_{k-1}^{N_l,l}(g_{k-1}P(\varphi))}{\eta_{k-1}^{N_l,l}(g_{k-1})} - 
\frac{\eta_{k-1}^{N_l,l-1}(g_{k-1}P(\varphi))}{\eta_{k-1}^{N_l,l-1}(g_{k-1})}\right\} -
\Bigg\{\frac{\eta_{k-1}^{l}(g_{k-1}P(\varphi))}{\eta_{k-1}^{l}(g_{k-1})} - \\ & &
\frac{\eta_{k-1}^{l-1}(g_{k-1}P(\varphi))}{\eta_{k-1}^{l-1}(g_{k-1})}\Bigg\}\Bigg]\Bigg|.
\end{eqnarray*}

$T_1$ and $T_2$ can be bounded using almost the same calculations, so we consider only $T_1$.
For $T_1$ this can easily be upper-bounded by $\sum_{j=1}^4 T_{j+3}$ where
\begin{eqnarray*}
T_4 & = &  \left|\mathbb{E}\left[
\Bigg\{\frac{1}{\eta_{k-1}^{N_l,l}(g_{k-1})}-\frac{1}{\eta_{k-1}^{l}(g_{k-1})}\Bigg\}
[\eta_{k-1}^{N_l,l}-\eta_{k-1}^{l}](g_{k-1}[P^l-P](\varphi))
\right]\right|\\
T_5 & = &  \left|\mathbb{E}\left[\frac{1}{\eta_{k-1}^{l}(g_{k-1})}
[\eta_{k-1}^{N_l,l}-\eta_{k-1}^{l}](g_{k-1}[P^l-P](\varphi))
\right]\right|\\
T_6 & = &  \left|\mathbb{E}\left[\eta_{k-1}^{l}(g_{k-1}[P^l-P](\varphi))
\Bigg\{\frac{1}{\eta_{k-1}^{N_l,l}(g_{k-1})\eta_{k-1}^{l}(g_{k-1})}-\frac{1}{\eta_{k-1}^{l}(g_{k-1})^2}\Bigg\}
[\eta_{k-1}^{l}-\eta_{k-1}^{N_l,l}](g_{k-1})
\right]\right|\\
T_7 & = &  \left|\mathbb{E}\left[
\frac{\eta_{k-1}^{l}(g_{k-1}[P^l-P](\varphi))}{\eta_{k-1}^{l}(g_{k-1})^2}[\eta_{k-1}^{l}-\eta_{k-1}^{N_l,l}](g_{k-1})
\right]\right|.
\end{eqnarray*}
For both $T_4$ and $T_6$, one can use the Cauchy-Schwarz inequality, along with \cite[Proposition C.6]{mlpf}, the lower-bound on $g_{k-1}$ and weak-error results for diffusions to show that
$$
\max\{T_4,T_6\} \leq \frac{C\overline{\|\varphi\|}\Delta_l}{N_l}.
$$
For $T_5$ and $T_7$ one can use standard bias bounds for particle filters (see e.g.~the proof of \cite[Lemma C.3]{mlpf}), the lower-bound on $g_{k-1}$ and weak-error results for diffusions to yield
$$
\max\{T_5,T_7\} \leq \frac{C\overline{\|\varphi\|}\Delta_l}{N_l}
$$
and hence
\begin{equation}\label{eq:part_conv_bias1}
\max\{T_1,T_2\} \leq \frac{C\overline{\|\varphi\|}\Delta_l}{N_l}.
\end{equation}

For $T_3$ using \cite[Lemma C.5]{mlpf} we have the upper-bound $T_3\leq\sum_{j=1}^6 T_{j+7}$ where
\begin{eqnarray*}
T_8 & = &  \left|\mathbb{E}\left[\frac{1}{\eta_{k-1}^{N_l,l}(g_{k-1})}
\left\{
[\eta_{k-1}^{N_l,l}-\eta_{k-1}^{N_l,l-1}](g_{k-1}P(\varphi)) -
[\eta_{k-1}^{l}-\eta_{k-1}^{l-1}](g_{k-1}P(\varphi))
\right\}\right]\right| \\
T_9 & = & \left|\mathbb{E}\left[\frac{\eta_{k-1}^{N_l,l-1}(g_{k-1}P(\varphi))}{\eta_{k-1}^{N_l,l}(g_{k-1})
\eta_{k-1}^{N_l,l-1}(g_{k-1})}
\left\{
[\eta_{k-1}^{N_l,l}-\eta_{k-1}^{N_l,l-1}](g_{k-1}) -
[\eta_{k-1}^{l}-\eta_{k-1}^{l-1}](g_{k-1})
\right\}\right]\right|\\
T_{10} & = & \left|\mathbb{E}\left[\frac{1}{\eta_{k-1}^{N_l,l}(g_{k-1})\eta_{k-1}^{l}(g_{k-1})}
\{[\eta_{k-1}^{l}-\eta_{k-1}^{N_l,l}](g_{k-1})\}\{[\eta_{k-1}^{l}-\eta_{k-1}^{l-1}](g_{k-1}P(\varphi))\}
\right]\right|\\
T_{11} & = & \left|\mathbb{E}\left[\frac{1}{\eta_{k-1}^{N_l,l}(g_{k-1})\eta_{k-1}^{N_l,l-1}(g_{k-1})}
\{[\eta_{k-1}^{N_l,l-1}-\eta_{k-1}^{l-1}](g_{k-1}P(\varphi))\}\{
[\eta_{k-1}^{l}-\eta_{k-1}^{l-1}](g_{k-1})
\}\right]\right|\\
T_{12} & = & \left|\mathbb{E}\left[\frac{\eta_{k-1}^{l-1}(g_{k-1}P(\varphi))}{
\eta_{k-1}^{l}(g_{k-1})\eta_{k-1}^{N_l,l-1}(g_{k-1})\eta_{k-1}^{l-1}(g_{k-1})}
\{[\eta_{k-1}^{N_l,l-1}-\eta_{k-1}^{l-1}](g_{k-1})\}
\{[\eta_{k-1}^{l}-\eta_{k-1}^{l-1}](g_{k-1})\}
\right]\right|\\
T_{13} & = & \left|\mathbb{E}\left[\frac{\eta_{k-1}^{l-1}(g_{k-1}P(\varphi))}{\eta_{k-1}^{N_l,l}(g_{k-1})
\eta_{k-1}^{l}(g_{k-1})\eta_{k-1}^{N_l,l-1}(g_{k-1})}
\{[\eta_{k-1}^{N_l,l}-\eta_{k-1}^{l}](g_{k-1})\}
\{[\eta_{k-1}^{l}-\eta_{k-1}^{l-1}](g_{k-1})\}
\right]\right|.
\end{eqnarray*}
Since  $T_8$ and $T_9$ can be bounded using similar approaches, we consider only $T_8$. Similarly
$T_{10},\dots,T_{13}$ can be bounded in almost the same way, so we consider only $T_{10}$.
For $T_8$ one has $T_8\leq T_{14}+T_{15}$ where
\begin{eqnarray*}
T_{14} & = &  \Bigg|\mathbb{E}\Bigg[\Bigg\{\frac{1}{\eta_{k-1}^{N_l,l}(g_{k-1})}-\frac{1}{\eta_{k-1}^{l}(g_{k-1})}\Bigg\}
\Bigg\{
[\eta_{k-1}^{N_l,l}-\eta_{k-1}^{N_l,l-1}](g_{k-1}P(\varphi)) - 
[\eta_{k-1}^{l}-\eta_{k-1}^{l-1}](g_{k-1}P(\varphi))
\Bigg\}\Bigg]\Bigg| \\
T_{15} & = &  \left|\mathbb{E}\left[\frac{1}{\eta_{k-1}^{l}(g_{k-1})}
\left\{
[\eta_{k-1}^{N_l,l}-\eta_{k-1}^{N_l,l-1}](g_{k-1}P(\varphi)) -
[\eta_{k-1}^{l}-\eta_{k-1}^{l-1}](g_{k-1}P(\varphi))
\right\}\right]\right|.
\end{eqnarray*}
For $T_{14}$ one can use Cauchy-Schwarz, the lower-bound on $g_{k-1}$, \cite[Proposition C.6]{mlpf}, 
and Lemma  \ref{lem:ml_conv1}. For $T_{15}$ we can apply the lower-bound on $g_{k-1}$ and the induction hypothesis. These two arguments give that
$$
\max\{T_{14},T_{15}\} \leq \frac{C\overline{\|\varphi\|}}{N_l}
\left(\Delta_l^{\tfrac{1}{2}}
+ \frac{\Delta_l^{\tfrac{1}{2}-\varepsilon}}{N_l}
\right)^{\tfrac{1}{2}}
$$
and hence that
$$
T_8 \leq \frac{C\overline{\|\varphi\|}}{N_l}
\left(\Delta_l^{\tfrac{1}{2}}
+ \frac{\Delta_l^{\tfrac{1}{2}-\varepsilon}}{N_l}
\right)^{\tfrac{1}{2}}.
$$
For $T_{10}$ we have the upper-bound $T_{10}\leq T_{16}+T_{17}$ where
\begin{eqnarray*}
T_{16} & = & \Bigg|\mathbb{E}\Bigg[\Bigg\{\frac{1}{\eta_{k-1}^{N_l,l}(g_{k-1})\eta_{k-1}^{l}(g_{k-1})}-\frac{1}{\eta_{k-1}^{l}(g_{k-1})^2}\Bigg\}
\{[\eta_{k-1}^{l}-\eta_{k-1}^{N_l,l}](g_{k-1})\}\times\\ & &\{[\eta_{k-1}^{l}-\eta_{k-1}^{l-1}](g_{k-1}P(\varphi))\}
\Bigg]\Bigg|\\
T_{17} & = & \left|\mathbb{E}\left[\frac{1}{\eta_{k-1}^{l}(g_{k-1})^2}
\{[\eta_{k-1}^{l}-\eta_{k-1}^{N_l,l}](g_{k-1})\}\{[\eta_{k-1}^{l}-\eta_{k-1}^{l-1}](g_{k-1}P(\varphi))\}
\right]\right|.
\end{eqnarray*}
For $T_{15}$ one can use Cauchy-Schwarz, the lower-bound on $g_{k-1}$, \cite[Proposition C.6]{mlpf} (twice) along with \cite[Lemma D.2]{mlpf}. For $T_{16}$ we can apply the lower-bound on $g_{k-1}$, \cite[Lemma D.2]{mlpf} and bias results for particle filters. These two arguments give
$$
\max\{T_{15},T_{16}\} \leq \frac{C\overline{\|\varphi\|}\Delta_l}{N_l}
$$
and thus
$$
T_9 \leq \frac{C\overline{\|\varphi\|}\Delta_l}{N_l}.
$$
Thus we have proved that
\begin{equation} \label{eq:part_conv_bias2}
T_2 \leq \frac{C\overline{\|\varphi\|}}{N_l}
\left(\Delta_l^{\tfrac{1}{2}}
+ \frac{\Delta_l^{\tfrac{1}{2}-\varepsilon}}{N_l}
\right)^{\tfrac{1}{2}}.
\end{equation}
The proof can be completed by combining \eqref{eq:part_conv_bias1} and \eqref{eq:part_conv_bias2}.
\end{proof}

\begin{rem}\label{rem:bias1}
Using the approach to the proof of Lemma \ref{lem:bias1} one can also prove the following result. Assume (A\ref{ass:diff1}-\ref{ass:g}). Then for any $k\in\mathbb{N}$ there exists a $C<+\infty$ such that for any $(l,N_l,\varepsilon,\varphi)\in\mathbb{N}^2\times(0,1/2)\times\mathcal{B}_b(\mathsf{X})\cap\textrm{\emph{Lip}}(\mathsf{X})$
$$
\left|\mathbb{E}\left[[\eta_k^{N_l,l}-\eta_k^{N_l,a}](\varphi)\right]\right| \leq \frac{C\overline{\|\varphi\|}}{N_l}
\left(\Delta_l^{\tfrac{1}{2}}
+ \frac{\Delta_l^{\tfrac{1}{2}-\varepsilon}}{N_l}
\right)^{\tfrac{1}{2}}.
$$
\end{rem}

%

\begin{lem}\label{lem:bias2}
Assume (A\ref{ass:diff1}-\ref{ass:g}). Then for any $(k,\varphi)\in\mathbb{N}\times
\mathcal{B}_b(\mathsf{X})\cap\mathcal{C}_b^2(\mathsf{X},\mathbb{R})$ there exists a $C<+\infty$ such that for any $(l,N_l,\varepsilon)\in\mathbb{N}^2\times(0,1/2)$
$$
\left|\mathbb{E}\left[[\tfrac{1}{2}\eta_k^{N_l,l}+\tfrac{1}{2}\eta_k^{N,l,a}-\eta_k^{N_l,l-1}](\varphi)-[\eta_k^l-\eta_k^{l-1}](\varphi)\right]\right| \leq \frac{C}{N_l}\Bigg\{\left(\Delta_l
+ \frac{\Delta_l^{\tfrac{1}{2}-\varepsilon}}{N_l}
\right)^{\tfrac{1}{2}} +
\left(\Delta_l^{\tfrac{1}{2}}
+ \frac{\Delta_l^{\tfrac{1}{2}-\varepsilon}}{N_l}
\right)
\Bigg\}.
$$
\end{lem}

\begin{proof}
The proof is by induction, with the case $k=1$ trivial as there is no bias, so we proceed to the induction step. Following the proof of Lemma \ref{lem:ml_conv2} we have the decomposition
$$
\left|\mathbb{E}\left[[\tfrac{1}{2}\eta_k^{N_l,l}+\tfrac{1}{2}\eta_k^{N,l,a}-\eta_k^{N_l,l-1}](\varphi)-[\eta_k^l-\eta_k^{l-1}](\varphi)\right]\right| \leq  T_1 + T_2
$$
where
\begin{eqnarray}
T_1 & = & \Bigg|\mathbb{E}\Bigg[
\Bigg\{\frac{\tfrac{1}{2}\eta_{k-1}^{N_l,l}(g_{k-1}[P^l-P](\varphi))}{\eta_{k-1}^{N_l,l}(g_{k-1})} + 
\frac{\tfrac{1}{2}\eta_{k-1}^{N_l,l,a}(g_{k-1}[P^l-P](\varphi))}{\eta_{k-1}^{N_l,l,a}(g_{k-1})}  - 
\frac{\eta_{k-1}^{N_l,l-1}(g_{k-1}[P^{l-1}-P](\varphi))}{\eta_{k-1}^{N_l,l-1}(g_{k-1})} 
\Bigg\}
- \nonumber\\ & &
\Bigg\{
\frac{\tfrac{1}{2}\eta_{k-1}^{l}(g_{k-1}[P^l-P](\varphi))}{\eta_{k-1}^{l}(g_{k-1})} + 
\frac{\tfrac{1}{2}\eta_{k-1}^{l}(g_{k-1}[P^l-P](\varphi))}{\eta_{k-1}^{l}(g_{k-1})}  - 
\frac{\eta_{k-1}^{l-1}(g_{k-1}[P^{l-1}-P](\varphi))}{\eta_{k-1}^{l-1}(g_{k-1})} 
\Bigg\}
\Bigg]\Bigg| \nonumber\\
T_2 & = & \Bigg|\mathbb{E}\Bigg[
\Bigg\{\frac{\tfrac{1}{2}\eta_{k-1}^{N_l,l}(g_{k-1}P(\varphi))}{\eta_{k-1}^{N_l,l}(g_{k-1})} + 
\frac{\tfrac{1}{2}\eta_{k-1}^{N_l,l,a}(g_{k-1}P(\varphi))}{\eta_{k-1}^{N_l,l,a}(g_{k-1})}  - 
\frac{\eta_{k-1}^{N_l,l-1}(g_{k-1}P(\varphi))}{\eta_{k-1}^{N_l,l-1}(g_{k-1})} 
\Bigg\}
- \nonumber\\ & &
\Bigg\{
\frac{\tfrac{1}{2}\eta_{k-1}^{l}(g_{k-1}P(\varphi))}{\eta_{k-1}^{l}(g_{k-1})} + 
\frac{\tfrac{1}{2}\eta_{k-1}^{l}(g_{k-1}P(\varphi))}{\eta_{k-1}^{l}(g_{k-1})}  - 
\frac{\eta_{k-1}^{l-1}(g_{k-1}P(\varphi))}{\eta_{k-1}^{l-1}(g_{k-1})} 
\Bigg\}\Bigg]\Bigg|.\label{eq:part_conv_bias4} 
\end{eqnarray}
For $T_1$ one can match the empirical and limit terms across the $l,a,l-1$ and adopt the same proof approach as used for $T_1$ in the proof of Lemma \ref{lem:bias1}; since the proofs will be repeated, we omit them and remark only that
\begin{equation}\label{eq:part_conv_bias3}
T_1 \leq \frac{C\Delta_l}{N_l}.
\end{equation}

For $T_2$ using Lemma \ref{lem:tech_lem_rat} and the triangular inequality we have that $T_2\leq\sum_{j=1}^4 T_{j+2}$ where
\begin{eqnarray*}
T_3 & = & \Bigg|\mathbb{E}\Bigg[\frac{1}{\eta_{k-1}^{N_l,l-1}(g_{k-1})}
[\tfrac{1}{2}\eta_{k-1}^{N_l,l}+\tfrac{1}{2}\eta_{k-1}^{N_l,l,a}-\eta_{k-1}^{N_l,l-1}](g_{k-1}P(\varphi)) -
\frac{1}{\eta_{k-1}^{l-1}(g_{k-1})}
[\tfrac{1}{2}\eta_{k-1}^{l}+\\ & & \tfrac{1}{2}\eta_{k-1}^{l}-\eta_{k-1}^{l-1}](g_{k-1}P(\varphi))\Bigg]\Bigg|\\
T_4 & = &  \Bigg|\mathbb{E}\Bigg[\frac{1}{\eta_{k-1}^{N_l,l}(g_{k-1})\eta_{k-1}^{N_l,l-1}(g_{k-1})}\tfrac{1}{2}
\{[\eta_{k-1}^{N_l,l}-\eta_{k-1}^{N_l,l,a}](g_{k-1}P(\varphi))\}\{[\eta_{k-1}^{N_l,l-1}-\eta_{k-1}^{N_l,l}](g_{k-1})\}
\Bigg]\Bigg| 
\end{eqnarray*}
\begin{eqnarray*}
T_{5} & = & \Bigg|\mathbb{E}\Bigg[
\frac{\tfrac{1}{2}\eta_{k-1}^{N_l,l,a}(g_{k-1}P(\varphi))}{\eta_{k-1}^{N_l,l,a}(g_{k-1})\eta_{k-1}^{N_l,l}(g_{k-1})\eta_{k-1}^{N_l,l-1}](g_{k-1})}
\{[\eta_{k-1}^{N_l,l,a}-\eta_{k-1}^{N_l,l}](g_{k-1})\}\{[\eta_{k-1}^{N_l,l-1}-\eta_{k-1}^{N_l,l}](g_{k-1})\}
\Bigg]\Bigg|\\
T_{6} & = & \Bigg|\mathbb{E}\Bigg[
\frac{\eta_{k-1}^{N_l,l,a}(g_{k-1}P(\varphi))}{\eta_{k-1}^{N_l,l,a}(g_{k-1})\eta_{k-1}^{N_l,l-1}(g_{k-1})}
[\tfrac{1}{2}\eta_{k-1}^{N_l,l}+\tfrac{1}{2}\eta_{k-1}^{N_l,l,a}-\eta_{k-1}^{N_l,l-1}](g_{k-1}) - 
\frac{\eta_{k-1}^{l}(g_{k-1}P(\varphi))}{\eta_{k-1}^{l}(g_{k-1})\eta_{k-1}^{l-1}(g_{k-1})}\times \\ & &
[\tfrac{1}{2}\eta_{k-1}^{l}+\tfrac{1}{2}\eta_{k-1}^{l}-\eta_{k-1}^{l-1}](g_{k-1})
\Bigg]\Bigg|.
\end{eqnarray*}
As $T_3$ (resp.~$T_4$) can be dealt with using similar arguments to $T_{6}$ (resp.~$T_{5}$) we only prove bounds for $T_3$ (resp.~$T_4$). For $T_3$ we have the upper-bound $T_3\leq\sum_{j=1}^4 T_{j+6}$ where
\begin{eqnarray*}
T_7 & = & \Bigg|\mathbb{E}\Bigg[\Bigg\{\frac{1}{\eta_{k-1}^{N_l,l-1}(g_{k-1})} -
\frac{1}{\eta_{k-1}^{l-1}(g_{k-1})} 
\Bigg\}
\Big\{
[\tfrac{1}{2}\eta_{k-1}^{N_l,l}+\tfrac{1}{2}\eta_{k-1}^{N_l,l,a}-\eta_{k-1}^{N_l,l-1}](g_{k-1}P(\varphi)) - \\ & & 
[\eta_{k-1}^{l}-\eta_{k-1}^{l-1}](g_{k-1}P(\varphi))
\Big\}
\Bigg]\Bigg|\\
T_8 & = & \Bigg|\mathbb{E}\Bigg[
\frac{1}{\eta_{k-1}^{l-1}(g_{k-1})}\Big\{
[\tfrac{1}{2}\eta_{k-1}^{N_l,l}+\tfrac{1}{2}\eta_{k-1}^{N_l,l,a}-\eta_{k-1}^{N_l,l-1}](g_{k-1}P(\varphi)) - 
[\eta_{k-1}^{l}-\eta_{k-1}^{l-1}](g_{k-1}P(\varphi))
\Big\} \Bigg]\Bigg|\\
T_9 & = & \Bigg|\mathbb{E}\Bigg[
[\eta_{k-1}^{l}-\eta_{k-1}^{l-1}](g_{k-1}P(\varphi))
\Bigg\{\frac{1}{\eta_{k-1}^{N_l,l-1}(g_{k-1})\eta_{k-1}^{l-1}(g_{k-1})} -
\frac{1}{\eta_{k-1}^{l-1}(g_{k-1})^2}
\Bigg\}\times\\ & &[\eta_{k-1}^{l-1}-\eta_{k-1}^{N_l,l-1}](g_{k-1})
\Bigg]\Bigg|\\
T_{10} & = & \Bigg|\mathbb{E}\Bigg[\frac{[\eta_{k-1}^{l}-\eta_{k-1}^{l-1}](g_{k-1}P(\varphi))}{\eta_{k-1}^{l-1}(g_{k-1})^2}
[\eta_{k-1}^{l-1}-\eta_{k-1}^{N_l,l-1}](g_{k-1})
\Bigg]\Bigg|
\end{eqnarray*}
For $T_7$ one can use Cauchy-Schwarz, the lower-bound on $g_{k-1}$, \cite[Proposition C.6]{mlpf} and Lemma \ref{lem:ml_conv2}. For $T_8$ we can apply the lower-bound on $g_{k-1}$ and the induction hypothesis. These arguments yield
$$
\max\{T_7,T_8\} \leq \frac{C}{N_l}\left(\Delta_l
+ \frac{\Delta_l^{\tfrac{1}{2}-\varepsilon}}{N_l}
\right)^{\tfrac{1}{2}}.
$$
For $T_9$ we can use \cite[Lemma D.2]{mlpf}, Cauchy-Schwarz, the lower-bound on $g_{k-1}$ and \cite[Proposition C.6]{mlpf} twice. For $T_{10}$ we can use \cite[Lemma D.2]{mlpf}, the lower-bound on $g_{k-1}$ and standard bias results for particle filters. This gives
$$
\max\{T_9,T_{10}\} \leq \frac{C\Delta_l}{N_l}.
$$
Collecting the above arguments, we have shown that
$$
T_3 \leq \frac{C}{N_l}\left(\Delta_l
+ \frac{\Delta_l^{\tfrac{1}{2}-\varepsilon}}{N_l}
\right)^{\tfrac{1}{2}}.
$$
For $T_4$ we have the upper-bound $T_4\leq\sum_{j=1}^3 T_{j+10}$ where
\begin{eqnarray*}
T_{11} & = & \Bigg|\mathbb{E}\Bigg[\frac{1}{\eta_{k-1}^{N_l,l}(g_{k-1})\eta_{k-1}^{N_l,l-1}(g_{k-1})}\tfrac{1}{2}
\{[\eta_{k-1}^{N_l,l}-\eta_{k-1}^{N_l,l,a}](g_{k-1}P(\varphi))\}\{[\eta_{k-1}^{N_l,l-1}-\eta_{k-1}^{N_l,l}](g_{k-1}) - \\ & & 
[\eta_{k-1}^{l-1}-\eta_{k-1}^{l}](g_{k-1})
\}
\Bigg]\Bigg| \\
T_{12} & = & \Bigg|\mathbb{E}\Bigg[
\tfrac{1}{2}
\{[\eta_{k-1}^{N_l,l}-\eta_{k-1}^{N_l,l,a}](g_{k-1}P(\varphi))\}[\eta_{k-1}^{l-1}-\eta_{k-1}^{l}](g_{k-1})
\Bigg\{
\frac{1}{\eta_{k-1}^{N_l,l}(g_{k-1})\eta_{k-1}^{N_l,l-1}(g_{k-1})}- \\ & & 
\frac{1}{\eta_{k-1}^{l}(g_{k-1})\eta_{k-1}^{l-1}(g_{k-1})}
\Bigg\}
\Bigg]\Bigg| \\
T_{13} & = & \Bigg|\mathbb{E}\Bigg[
\frac{[\eta_{k-1}^{l-1}-\eta_{k-1}^{l}](g_{k-1})}{\eta_{k-1}^{l}(g_{k-1})\eta_{k-1}^{l-1}(g_{k-1})}
\{[\eta_{k-1}^{N_l,l}-\eta_{k-1}^{N_l,l,a}](g_{k-1}P(\varphi))\}
\Bigg]\Bigg|.
\end{eqnarray*}
For $T_{11}$ one can use the lower-bound on $g_{k-1}$, Cauchy-Schwarz, Remark \ref{rem:ml_conv1} and Lemma 
\ref{lem:ml_conv2}. For $T_{12}$ we can apply \cite[Lemma D.2]{mlpf}, Cauchy-Schwarz, the lower-bound on $g_{k-1}$ and \cite[Proposition C.6]{mlpf} and Remark \ref{rem:ml_conv1}. For $T_{13}$ we can use he lower-bound on $g_{k-1}$, Remark \ref{rem:bias1} and  \cite[Lemma D.2]{mlpf}. Putting these results together, we have established that
$$
\max\{T_{11},T_{12},T_{13}\} \leq  \frac{C}{N_l}\Bigg\{\left(\Delta_l
+ \frac{\Delta_l^{\tfrac{1}{2}-\varepsilon}}{N_l}
\right)^{\tfrac{1}{2}} +
\left(\Delta_l^{\tfrac{1}{2}}
+ \frac{\Delta_l^{\tfrac{1}{2}-\varepsilon}}{N_l}
\right)
\Bigg\}
$$
and thus we can deduce the same (up-to a constant) upper-bound for $T_4$.
As a result of the above arguments, we have established that
\begin{equation}\label{eq:part_conv_bias5}
T_2 \leq  \frac{C}{N_l}\Bigg\{\left(\Delta_l
+ \frac{\Delta_l^{\tfrac{1}{2}-\varepsilon}}{N_l}
\right)^{\tfrac{1}{2}} +
\left(\Delta_l^{\tfrac{1}{2}}
+ \frac{\Delta_l^{\tfrac{1}{2}-\varepsilon}}{N_l}
\right)
\Bigg\}.
\end{equation}
The proof can be completed by combining \eqref{eq:part_conv_bias3} and \eqref{eq:part_conv_bias5}.
\end{proof}

\begin{theorem}\label{theo:bias}
Assume (A\ref{ass:diff1}-\ref{ass:g}). Then for any $(k,\varphi)\in\mathbb{N}\times
\mathcal{B}_b(\mathsf{X})\cap\mathcal{C}_b^2(\mathsf{X},\mathbb{R})$ there exists a $C<+\infty$ such that for any $(l,N_l,\varepsilon)\in\mathbb{N}^2\times(0,1/2)$
$$
\left|\mathbb{E}\left[[\pi_k^l-\pi_{k}^{l-1}]^{N_l}(\varphi) - 
[\pi_k^l-\pi_{k}^{l-1}](\varphi)\right]\right| \leq \frac{C}{N_l}\Bigg\{\left(\Delta_l
+ \frac{\Delta_l^{\tfrac{1}{2}-\varepsilon}}{N_l}
\right)^{\tfrac{1}{2}} +
\left(\Delta_l^{\tfrac{1}{2}}
+ \frac{\Delta_l^{\tfrac{1}{2}-\varepsilon}}{N_l}
\right)
\Bigg\}.
$$
\end{theorem}

\begin{proof}
This an be proved using Lemma \ref{lem:bias2}, Remark \ref{rem:bias1} and the approach used to deal with \eqref{eq:part_conv_bias4} in the proof of Lemma \ref{lem:bias2}. The proof is omitted for brevity.
\end{proof}

\end{document}